\documentclass[10pt,twoside,reqno]{amsart}
\usepackage[foot]{amsaddr}

\usepackage{amssymb,amsfonts,amsmath,amsthm,bm}
\usepackage[mathscr]{eucal}
\usepackage{bbm}
\usepackage{stmaryrd}
\usepackage{fullpage}
\usepackage{blkarray}
\usepackage{array}
\usepackage{mathdots}
\usepackage{arydshln}
\usepackage{mathdots}

\usepackage{etoolbox}
\patchcmd{\section}{\scshape}{\bfseries}{}{}
\makeatletter
\renewcommand{\@secnumfont}{\bfseries}
\makeatother

\theoremstyle{definition}
\newtheorem{thm}{Theorem}
\newtheorem*{thm*}{Theorem}
 
\newtheorem{lem}[thm]{Lemma}
\newtheorem{defi}[thm]{Definition} 
\newtheorem{cor}[thm]{Corollary}
\newtheorem{rem}[thm]{Remark}

\newtheorem{conj}[thm]{Conjecture}

\numberwithin{equation}{section}
\numberwithin{thm}{section}

\usepackage{enumerate}
\usepackage[shortlabels]{enumitem}

\usepackage{hyperref}
\usepackage[dvipsnames]{xcolor}
\newcommand\myshade{85}
\colorlet{mylinkcolor}{red}
\colorlet{mycitecolor}{blue}
\colorlet{myurlcolor}{Aquamarine}
\hypersetup{
	linkcolor  = mylinkcolor,
	citecolor  = mycitecolor,
	urlcolor   = myurlcolor!\myshade!black,
	colorlinks = true,
}

\usepackage{pythonhighlight}

\usepackage{tikz}

\usetikzlibrary{arrows,automata}
\usepackage[all,cmtip]{xy}

\newcommand{\la}[1]{\mathfrak{#1}}

\newcommand{\ZZ}{\mathbb{Z}}

\newcommand{\cC}{\mathcal{C}}

\DeclareMathOperator{\ch}{ch}

\DeclareMathOperator{\typeA}{A}

\newcommand{\qbin}[2]{\begin{bmatrix} #1 \\ #2 \end{bmatrix}}

\allowdisplaybreaks

\begin{document}

\date{}

\title{Completing the $\typeA_2$ Andrews--Schilling--Warnaar identities}
\author{Shashank Kanade}
\address{Department of Mathematics, University of Denver, Denver, CO 80208}
\email{shashank.kanade@du.edu}

\author{Matthew C.\ Russell}
\address{Department of Mathematics, 
University of Illinois Urbana-Champaign,
Urbana, IL 61801}
\email{mcr39@illinois.edu}

\begin{abstract}
We study the Andrews--Schilling--Warnaar sum-sides for the principal characters of standard (i.e., integrable, highest weight) modules  of $\typeA_2^{(1)}$.
These characters have been studied recently by various subsets of Corteel, Dousse, Foda, Uncu, Warnaar and Welsh.	
We prove complete sets of identities for moduli $5$ through $8$ and $10$, in Andrews--Schilling--Warnaar form. 
The cases of moduli $6$ and $10$ are new. 
Our methods depend on the Corteel--Welsh recursions governing the cylindric partitions and on 
certain relations satisfied by the Andrews--Schilling--Warnaar sum-sides.
We speculate on the role of the latter in the proofs of higher modulus identities.
Further, we provide a complete set of conjectures for modulus $9$.
In fact, we show that at any given modulus, a complete set of conjectures may be deduced using a subset of ``seed'' conjectures.
These seed conjectures are obtained by appropriately truncating conjectures for the ``infinite'' level.
Additionally, for moduli $3k$, we use an identity of Weierstra{\ss}   to deduce new sum-product identities 
starting from the results of Andrews--Schilling--Warnaar.
\end{abstract}
\maketitle

\renewcommand\contentsname{\textbf{Contents}}
\setcounter{tocdepth}{1}
\tableofcontents

\section{Introduction}

Characters of various structures related to the modules of affine Lie algebras, and more generally, vertex operator algebras, are a tremendously rich 
source of various kinds of combinatorial, number-theoretic and $q$-series-theoretic identities. 
The landscape of this subject is truly vast, and we shall restrict ourselves to the scenarios that lead to the identities of the Rogers--Ramanujan-type:
\begin{align}
\sum_{n\geq 0}\dfrac{q^{n^2}}{(q)_n} &= \prod_{m\geq 0}\dfrac{1}{(1-q^{5m+1})(1-q^{5m+4})}\label{eqn:RR1}\\
\sum_{n\geq 0}\dfrac{q^{n^2+n}}{(q)_n} &= \prod_{m\geq 0}\dfrac{1}{(1-q^{5m+2})(1-q^{5m+3})}\label{eqn:RR2}.
\end{align}

Here and below, we use the standard notation for $n\in\ZZ_{\geq 0}\cup\{\infty\}$:
\begin{align}
(a;q)_n=\prod_{0\leq t < n} (1-aq^t),
\end{align}
and simply writing $(a)_n$ when the base $q$ is understood.
We will also use the following notation:
\begin{align}
\theta(a;q)&=(a; q)_\infty(q/a; q)_\infty\\
\theta(a_1,a_2,\cdots,a_k;q)&=\theta(a_1;q)\theta(a_2;q)\cdots \theta(a_k;q).
\end{align}

Throughout this paper, we will be viewing identities such as \eqref{eqn:RR1} and \eqref{eqn:RR2} as purely formal power series 
identities. We shall refer to the periodicity in the product-side as the ``modulus'' of the identity, for instance, we will say
that the Rogers-Ramanujan identities above \eqref{eqn:RR1} and \eqref{eqn:RR2} are mod $5$ identities.

We recall the combinatorial interpretations of \eqref{eqn:RR1} and \eqref{eqn:RR2}:
The LHS of \eqref{eqn:RR1} is naturally seen as the generating function of partitions where adjacent parts differ by at least $2$,
and the RHS of \eqref{eqn:RR1} is the generating function of partitions where each part is $\equiv \pm 1\pmod{5}$.
The LHS of \eqref{eqn:RR2} is the generating function of partitions where adjacent parts differ by at least $2$ and where $1$ is not allowed as a part,
and the RHS of \eqref{eqn:RR2} is the generating function of partitions where each part is $\equiv \pm 2\pmod{5}$.
Lepowsky and Wilson (see \cite{LepWil-RR}, \cite{LepWil-struI} and \cite{LepWil-struII}) were the first to prove these combinatorial identities using representation theory of $\typeA_1^{(1)}$ at level $3$, and to connect Andrews--Gordon \cite[Ch.\ 7]{And-book} and Andrews--Bressoud \cite{And-br}, \cite{Bre-identities} identities to standard
modules of $\typeA_1^{(1)}$. Meurman and Primc \cite{MeuPri-annide} then proved these higher level identities representation-theoretically.
Since then, this subject has grown immensely, and various kinds of deep identities related to affine Lie algebras continue to be found and conjectured. As two recent examples, we only mention
\cite{GriOnoWar} and \cite{CapMeuPriPri}, and for some more general background, refer the reader to the excellent book by Sills, \cite{Sil-book}.
 
For $\la{g}$ an affine Kac--Moody Lie algebra and a highest weight module $M$ with highest weight $\lambda$, the principally specialized character of $M$ is:
\begin{align}
\chi(M) = \left.\left(e^{-\lambda}\cdot\ch(M)\right)\right\vert_{e^{-\alpha_0},\dots,e^{-\alpha_r} \mapsto q},
\end{align}
where $\ch(M)$ is the usual character of $M$, so that $\chi(M)$ is a formal power series in $\ZZ[[e^{-\alpha_0},\dots,e^{-\alpha_r}]]$.
Up to the factor $F={(q;q^2)_\infty^{-1}}$, principally specially characters of irreducible standard 
(i.e., integrable and highest weight) modules of $\typeA_1^{(1)}$
lead to the Andrews--Gordon and the Andrews--Bressoud identities. 
Due to the work of Lepowsky and Wilson, this factor $F$ can be seen as the character of the principal Heisenberg subalgebra of $\typeA_1^{(1)}$.
We now introduce the terminology commonly used for this procedure of deleting the factor $F$ from the character -- 
by principal character of a highest weight module $M$, we mean the principally specialized character of $M$ divided by the appropriate
factor $F$ which depends solely on $\la{g}$. If $M$ is an irreducible standard module $L(\lambda)$, we shall use $\chi(\Omega(\lambda))$ to denote 
its principal character.
In this notation, \eqref{eqn:RR1} and \eqref{eqn:RR2} are related to $\la{g}=\typeA_1^{(1)}$, with
the RHS of \eqref{eqn:RR1} being $\chi(\Omega(2\Lambda_0+\Lambda_1))=\chi(\Omega(\Lambda_0+2\Lambda_0))$
and the RHS of \eqref{eqn:RR2} being $\chi(\Omega(3\Lambda_0))=\chi(\Omega(3\Lambda_1))$, as can be calculated using Weyl-Kac character formula
and Lepowsky's numerator formula.

It is now natural to ask what happens for the principal characters of standard modules for the higher rank affine Lie algebras. 
The short answer is that no general answer is known.
However, in 1999, Andrews, Schilling and Warnaar in their groundbreaking paper \cite{AndSchWar} invented an $\typeA_2$ generalization of the ($\typeA_1$) Bailey lemma and found 
$q$-series sum-sides (analogous to the LHS of \eqref{eqn:RR1} and \eqref{eqn:RR2}) for many, but not all, principal characters of standard modules 
for $\typeA_2^{(1)}$.
For example, one of the identities of \cite{AndSchWar} related to the level $3$ standard modules for $\typeA_2^{(1)}$ reads:
\begin{align}
{(q)_\infty}
\sum_{r,s\geq 0}
&
\dfrac{q^{r^2-rs+s^2+r+s}}{(q)_{r+s}(q)_{r+s+1}}\qbin{r+s}{r}_{q^3}
=
\dfrac{1}{\theta(q^2,q^3;\,q^6)}
=\chi(\Omega(3\Lambda_0)).
\end{align}
However, note the crucial fact that due to the presence of the factor $(q)_\infty$,
the LHS does not represent a manifestly non-negative sum-side for the principal character.
Compare this with the situation of say, \eqref{eqn:RR1}, \eqref{eqn:RR2}.
Regardless, the identities given in \cite{AndSchWar} were a major achievement.
See \cite{War-A2further} for further illuminating insights on \cite{AndSchWar}.
See also \cite{War-HL} where Warnaar uses Hall--Littlewood polynomials to reprove the 
$\typeA_2$ Bailey lemma and many of the sum-product identities in \cite{AndSchWar}.

Let $\ell$ be the level corresponding to a standard module of $\typeA_2^{(1)}$. 
Then, the modulus of the corresponding identity is $m=\ell+3$. Choose the $k$ such that $m\in 3k+\{1,0,-1\}$.
Now, at level $\ell$, there are roughly as many distinct principal characters as the number of partitions of $\ell$
with at most $3$ parts. In general, for type $\typeA_r^{(1)}$, this is the number of compositions of $\ell$ with at most $r+1$ parts up to a dihedral symmetry, however for $\typeA_2^{(1)}$, 
we have $D_3=S_3$.
For $\typeA_2^{(1)}$, this number grows as a quadratic in $\ell$ (or $k$), however, the number of identities found in \cite{AndSchWar} is equal to $k+2$ if $3\nmid m$,
or $k$ if $3\mid m$. This means that a major subset of identities is missing from \cite{AndSchWar}.
In fact, the first time this happens is for level $3$, where one identity is missing, which we discover and prove below, see \eqref{eqn:210}:
\begin{align}
\sum_{r,s\geq 0}
&
\dfrac{q^{r^2-rs+s^2+s}}{(q)_{r+s}(q)_{r+s+1}}\qbin{r+s}{r}_{q^3}
=
\dfrac{1}{(q)_\infty}\dfrac{1}{\theta(q,q^2;\,q^6)}
=\dfrac{1}{(q)_\infty}\chi(\Omega(2\Lambda_0+\Lambda_1)).
\end{align}

In \cite{AndSchWar} itself, the authors showed how to transform their sum-sides to manifestly non-negative ones in the case of level $4$ (i.e., mod $7$ identities).
The missing identity at this level was first conjectured by Feigin, Foda and Welsh \cite{FeiFodWel} in the manifestly non-negative form and proved by Corteel and Welsh \cite{CorWel}.
Recently, Corteel, Dousse and Uncu \cite{CorDouUnc} gave the complete set of identities in the manifestly non-negative form for level $5$ (i.e., mod $8$ identities).
Even more recently, Warnaar \cite{War} added more identities at every level not divisible by $3$.
Still, the problem of completing the set of Andrews--Schilling--Warnaar identities remains, and the present paper is a step in the direction of answering this problem.

In \cite{FodWel}, Foda and Welsh studied the principal characters of standard modules of $\typeA_r^{(1)}$ (which can be alternately seen as characters of modules of $\mathcal{W}_{r+1}$ vertex operator algebras at a family of specific central charges) using cylindric partitions. Cylindric partitions were in fact first studied by Gessel and Krattenthaler in \cite{GesKra}, and they had already noticed the role of the affine Weyl group of $\typeA_r^{(1)}$ controlling their combinatorics. Actually, even in \cite{AndSchWar}, one of the $\typeA_2$ Bailey pairs (related to levels divisible by $3$) comes from the work of Gessel and Krattenthaler. 

Cylindric partitions are indeed extremely well-suited for this study, see for instance \cite{DJKMO}, \cite{JMMO} and \cite{Tin-crystal}. Let $c=(c_0,c_1,\cdots,c_r)$ denote a composition (possibly containing zeros) of $\ell$. 
After Foda and Welsh's analysis \cite{FodWel}, Corteel and Welsh \cite{CorWel} considered the maximum part statistic for the cylindric partitions. 
Let $F_c(z,q)$ be the generating function for the cylindric partitions of profile $c$, where $z$ corresponds to the maximum part and $q$ with total weight of the partition.
Corteel and Welsh discovered a system of recurrences that tie together $F_c(z,q)$ for various compositions of the same length of a given level $\ell$. 
On the other hand, due to Borodin's product formula \cite{Bor}, we know that $F_c(1,q)$ is an infinite product which equals the principal character of the  $\typeA_r^{(1)}$ module $L(c_0\Lambda_0+\cdots+c_r\Lambda_r)$
up to a factor of $(q)_\infty$.
Thus, cylindric partitions form the main engine on which the proofs in \cite{CorWel} and \cite{CorDouUnc} are based.

Now we come to the results and conjectures of the present paper.

Firstly, based on a crucial input by Warnaar \cite{War-email}, we enhace the Andrews--Schilling--Warnaar sum-sides by introducing the variable $z$ corresponding to the maximum part statistic of the cylindric partitions. 

Next, for any integrable level $\ell$ of $\typeA_2^{(1)}$, we present explicit conjectures for 
\begin{align}
H_c(z,q)=\dfrac{(zq)_\infty}{(q)_\infty} F_c(z,q)
\end{align}
for a large subset of compositions $c=(c_0,c_1,c_2)$ of $\ell$. Let us now describe this subset.
Due to the inherent cyclic symmetry in the definition of cylindric partitions, we know that $F_{(c_0,c_1,\cdots,c_r)}(z,q)=F_{(c_1,c_2,\cdots,c_r,c_0)}(z,q)$. We may thus arrange $c=(c_0,c_1,c_2)$ so
that $c_0$ is the biggest part. Now, our Conjecture \ref{conj:main} provides an explicit formula for $H_c(z,q)$ in terms of sums in the Andrews--Schilling--Warnaar form whenever $c_1,c_2\leq k-1$.
These are our ``seed'' conjectures.
For instance, in the example of level $\ell=16$ presented below in \eqref{eqn:level16}, this covers all compositions above the horizontal line.
In Theorem \ref{thm:remainingconj}, we then show that the Corteel--Welsh recursions are enough to determine the remaining $H_c(z,q)$ (i.e., for $c$ below the horizontal line).

Investigating deeper in the case of moduli $3k$ (with $k\geq 3$), we use another circle of ideas to prove new sum-product identities,
thereby enlarging the set of known identities in \cite{AndSchWar}.
Here, we observe that after setting $z=1$, all identities where $c_1\geq k$ can be proved to have precise formulas
in terms of the seed conjectures. This proof requires an identity of Weierstra\ss; see Lemma \ref{lem:weierstrass}.
Notably, at $z=1$, a subset of the seed identities was already established in \cite{AndSchWar}, and using these, we
actually prove a few more new sum-product identities.
For instance, at mod $9$, we prove the following new identity \eqref{eqn:mod9missing}:
\begin{align}
\sum_{r_1,r_2,s_1,s_2\geq 0}
\dfrac{q^{r_1^2-r_1s_1+s_1^2+r_2^2-r_2s_2+s_2^2+r_2+s_2}(1-2q^{1+r_1+s_1})}
{(q)_{r_1-r_2}(q)_{s_1-s_2}(q)_{r_2+s_2}(q)_{r_2+s_2+1}}
\qbin{r_2+s_2}{r_2}_{q^3}
=\dfrac{1}{(q)_\infty}\dfrac{1}{\theta(q,q^2,q^2,q^3,q^3;\,q^9)_\infty}.
\end{align}
Somehow, this relation of Weierstra{\ss} that we use seems quite relevant in studying characters of affine Lie algebras, as we had already
observed in the case of level $2$ modules for $\typeA_9^{(2)}$, \cite{KanRus-A92}.

At all moduli, the sum-product versions (obtained after setting $z=1$ and using character formulas to get products) of our seed conjectures can be seen to be appropriate truncations of $\typeA_2^{(1)}$ ``infinite level'' sum-product conjectures, which we provide in Section \ref{sec:infA2}. A subset of these conjectures was 
already proved by Warnaar in \cite{War}. In a way, infinite level, i.e., $\ell=\infty$ means that the affine singular vector $x_\theta(-1)^{\ell+1}$ gets pushed down to infinity and thus vanishes. Consequently, we may alternately view our infinite level conjectures also as corresponding to the principal characters of parabolically induced Verma modules which are also known as generalized Verma modules, see \cite{LepLi-book}.

Actually, as a warm up, in Section \ref{sec:infA1}, we complete the set of infinite level identities for $\typeA_1^{(1)}$ provided by Warnaar \cite{War} to include principal characters of generalized Verma modules induced from all finite dimensional $\mathfrak{sl}_2$-modules. In purely combinatorial terms, we prove infinite level analogues of all Andrews--Gordon (or Andrews--Bressoud) identities:
\begin{align}
\sum_{n_1,n_2,\cdots}
\dfrac{q^{n_1^2+n_2^2+\cdots\,\,\,+n_t+n_{t+1}+\cdots}}{(q)_{n_1-n_2}(q)_{n_2-n_3}\cdots }
=
\sum_{n_1,n_2,\cdots}
\dfrac{1}{(q)_{n_1}}
\left(
\prod_{1\leq i < t} q^{n_i^2}
\qbin{n_i}{n_{i+1}}
\prod_{j \geq t}q^{n_j^2+n_j}
\qbin{n_j}{n_{j+1}}
\right)
=\dfrac{1-q^t}{(q)_\infty}.
\end{align}
In fact, this identity, and its finitization that we provide below in \eqref{eqn:truncinfiniteAG}
were first proved by Warnaar in \cite{War-HL} using Hall--Littlewood polynomials; our methods are different. 
In Section \ref{sec:infA1}, we shall further deduce an identity (see Corollary \ref{cor:initcond}) 
that will be crucially used in latter sections.

Now we come to our proofs of complete sets of identities for moduli $5, 6, 7, 8$ and $10$.

We find and prove certain relations satisfied by the $(z,q)$ sums in Andrews--Schilling--Warnaar form. 
By explicit computations, we then show that the Corteel--Welsh system of recurrences at each of these levels 
is a consequence of these relations, thereby showing that our formulas of $H_c(z,q)$ presented in Conjecture \ref{conj:main} 
are indeed correct.
In Appendix \ref{app:verify}, we detail a simple and fast 
SAGE program that verifies all these proofs very quickly.
This program amounts to nothing more than checking that these recurrences are indeed equal to 
explicitly provided linear combinations of known relations.
This now provides alternate proofs of results in the moduli $5, 7, 8$.
The results in moduli $6$ and $10$ are brand new. 
After setting $z\mapsto 1$ and using the product formula, we can now complete the set of Andrews--Schilling--Warnaar
sum-product identities in moduli $6$ and $10$. 
For instance, the following mod $10$ identity corresponding to the profile $(3,3,1)$ that emerges as a consequence is new:
\begin{align}
\sum_{r_1,r_2,s_1,s_2\geq 0}
&\dfrac{q^{r_1^2-r_1s_1+s_1^2+r_2^2-r_2s_2+s_2^2}
(
q^{-r_1+s_2}-q^{r_2}+q^{1+r_1+r_2+s_2}-q^{1-r_1+r_2+s_1+s_2}-q^{s_1+s_2}
)
}{(q)_{r_1-r_2}(q)_{s_1-s_2}(q)_{r_2}(q)_{s_2}(q)_{r_2+s_2+1}}
\notag\\
&=
\dfrac{1}{(q)_\infty}\dfrac{1}{\theta(q,q,q^2,q^3,q^3,q^5;\,q^{10})}
=\dfrac{1}{(q)_\infty}\chi(\Omega(3\Lambda_0+3\Lambda_1+\Lambda_2)).
\end{align}

For the moduli where we do not have a proof, we conjecture that the relations we have found 
(detailed in Section \ref{sec:higher})
are enough to ultimately 
yield proofs, however, we mention that proving this by explicit calculations may be 
computationally intensive even in moduli $9$ or $11$.

\subsection*{Acknowledgments}
We are immensely grateful to S.\ Ole Warnaar for providing the crucial insight that the 
Andrews--Schilling--Warnaar sum-sides are compatible with the maximum part statistic of 
the cylindric partitions, for supplying Lemma \ref{lem:war_gen}, and for his numerous 
suggestions at various stages of this project.
We additionally thank Ali Uncu for carefully reading the manuscript.
We have benefitted greatly from illuminating discussions with  Sylvie Corteel, Jehanne Dousse, 
Ali Uncu
and Trevor Welsh.
SK is presently supported by Simons Collaboration Grant for Mathematicians \#636937.

\section{Infinite level \texorpdfstring{$\typeA_1^{(1)}$}{A1}}
\label{sec:infA1}
\subsection{Andrews--Gordon identities}
Recall the (combinatorial) Andrews--Gordon identities, \cite[Ch.\ 7]{And-book}.
\begin{thm}
\label{thm:AGcomb}
Let $k\geq 2$ and $1\leq t \leq k$.
The number of partitions of $n$ 
with difference at least $2$ at distance $k-1$ such that
$1$ appears at most $t-1$ times is
the same as number of partitions of $n$ 
into parts not congruent to $0$, $t$ or $2k+1-t$ modulo $2k+1$.
\end{thm}

The generating function form of these identities is:

\begin{thm}
\label{thm:AGq}
Let $k\geq 2$ and $1\leq t \leq k$.
We have:
\begin{align}
\sum_{n_1,n_2,\cdots,n_{k-1}}
\dfrac{q^{n_1^2+n_2^2+\cdots+n_{k-1}^2+n_t+n_{t+1}+\cdots+n_{k-1}}}{(q)_{n_1-n_2}(q)_{n_2-n_3}\cdots (q)_{n_{k-2}-n_{k-1}}(q)_{n_{k-1}}}
=\dfrac{(q^t,q^{2k+1-t},q^{2k+1};\,\,q^{2k+1})}{(q)_\infty}.
\end{align}
\end{thm}

\subsection{Infinite modulus}

Let $k=\infty$ in Theorem \ref{thm:AGcomb}. Clearly, the restriction 
of having difference at least $2$ at distance $k-1$ now becomes vacuous, and on the LHS
we are counting partitions of $n$ where $1$ appears at most $t-1$ times.
On the RHS, the modulus $2k+1$ is now infinite and thus vacuous.
Thus, we are left with counting partitions where $t$ does not appear as a part.
These two classes of partitions are obviously equinumerous. Indeed, looking at their complements, 
partitions of $n$ in which we \emph{do} have at least $t$ ones is in bijection
with partitions in which $t$ \emph{does} appear as a part.

Thus, the combinatorial versions of the Andrews--Gordon identities with $k=\infty$ are trivial.

However, looking at the $q$--series versions, the identities take the shape ($t\in\ZZ_{>0}\cup\{\infty\}$):
\begin{align}
\sum_{n_1,n_2,\cdots}
\dfrac{q^{n_1^2+n_2^2+\cdots\,\,\,+n_t+n_{t+1}+\cdots}}{(q)_{n_1-n_2}(q)_{n_2-n_3}\cdots }
=
\sum_{n_1,n_2,\cdots}
\dfrac{1}{(q)_{n_1}}
\left(
\prod_{1\leq i < t} q^{n_i^2}
\qbin{n_i}{n_{i+1}}
\prod_{j \geq t}q^{n_j^2+n_j}
\qbin{n_j}{n_{j+1}}
\right)
=\dfrac{1-q^t}{(q)_\infty}.
\label{eqn:InfiniteAG}
\end{align}
Note that for $t=\infty$ we take $q^\infty=0$.
This actually admits a two variable generalization which we shall prove below.

\subsection{Representation-theoretic interpretation}

Another way to look at the RHS of \eqref{eqn:InfiniteAG} is that this is also the 
principal character of  generalized (i.e., parabolically induced) $\typeA_1^{(1)}$ Verma module
at \emph{any} level $\ell$ induced from the irreducible $\mathfrak{sl}_2$ module $L( (t-1) \Lambda_1)$
for $t\in\ZZ_{>0}$.
When $t=\infty$, this is to be thought of as the 
principal character of \emph{any} $\typeA_1^{(1)}$ Verma module (in the usual sense).

\subsection{Proof and other related identities}
If one takes $t=\infty$ in \eqref{eqn:InfiniteAG}, we see:
\begin{align}
\sum_{n_1,n_2,\cdots}
\dfrac{1}{(q)_{n_1}}
\left(
\prod_{i\geq 1} q^{n_i^2}
\qbin{n_i}{n_{i+1}}
\right)
=\dfrac{1}{(q)_\infty}.
\label{eqn:InfiniteAG_l_inf}
\end{align}
This identity can be proved immediately by recognizing the left hand side as the 
generating function of partitions calculated using successive Durfee squares, see Figure \ref{fig:andrewswarnaar_inf}.
\begin{figure}
\begin{tikzpicture}
\draw (0,10) -- (3,10) -- (3,7) -- (0,7) -- (0,10);
\draw (0,7) -- (2,7) -- (2,5) -- (0,5) -- (0,7); 
\draw (0,5) -- (1,5) -- (1,4) -- (0,4) -- (0,5); 
\draw (-0.25,8.5) node{$n_1$};
\draw (1.5,10.25) node{$n_1$};
\draw (-0.25,6) node{$n_2$};
\draw (1,7.25) node{$n_2$};
\draw (-0.25,4.5) node{$n_3$};
\draw (0.5,5.25) node{$n_3$};
\draw[dashed] (3,10) -- (6,10) -- (6,9) -- (5,9) -- (5,8) -- (4,8) -- (4,7)-- (3,7);
\draw[dashed] (3,7) -- (3,6) -- (2.5,6) -- (2.5,5) 
  -- (2,5);
\draw[dashed] (2,5) -- (2,4.5) -- (1.5,4.5) -- (1.5,4)--(1,4);
\draw (1.5,8.5) node{$\mathbf{z^{n_1}q^{n_1^2}}$};
\draw (1,6) node{$\mathbf{z^{n_2}q^{n_2^2}}$};
\draw (5.5,8) node{$\mathbf{\dfrac{1}{(q)_{n_1}}}$};
\draw (3.5,6) node{$\qbin{\mathbf{n_1}}{\mathbf{n_2}}$};
\draw[dotted, thick] (0.5,3.5)	-- (0.5,2);
\end{tikzpicture}
\caption{Successive Durfee squares}
\label{fig:andrewswarnaar_inf}
\end{figure}
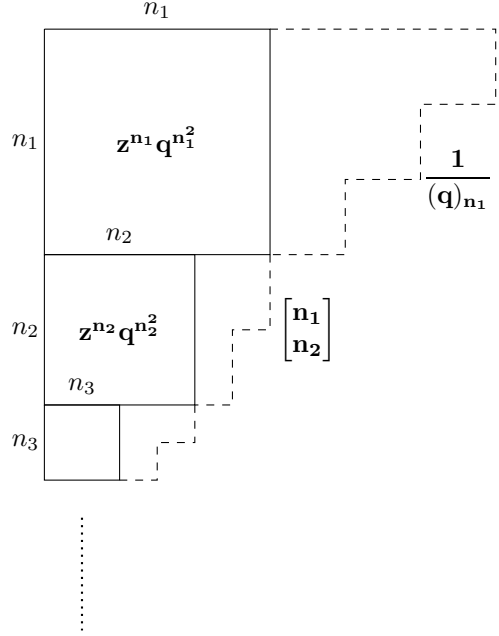
This interpretation in fact proves a $(z,q)$ version of this identity:
\begin{align}
\sum_{n_1,n_2,\cdots}
\dfrac{1}{(q)_{n_1}}
\left(
\prod_{i\geq 1} z^{n_i}q^{n_i^2}
\qbin{n_i}{n_{i+1}}
\right)
=\dfrac{1}{(zq)_\infty}.
\label{eqn:andrewswarnaar_inf}
\end{align}
Warnaar proves this identity in his paper \cite{War}. Let us review his proof
as it will help us in the case for all $t$.

Warnaar considers
\begin{align}
\Phi_n(z;q)=\dfrac{1}{(zq)_n}
\end{align}
and iterates the relation:
\begin{align}
\sum_{j=0}^n  z^jq^{j^2}
\qbin{n}{j}
\Phi_j(z;q)=\Phi_n(z;q).
\label{eqn:basicbin}
\end{align}
If we iterate it finitely many times, we reach:
\begin{align}
\sum_{n_1,n_2,\cdots, n_k=0}^{n_0}
z^{n_1+\cdots+n_k}q^{n_1^2+\cdots+n_k^2}
\qbin{n_0}{n_1}\qbin{n_1}{n_2}\cdots \qbin{n_{k-1}}{n_k}
\dfrac{1}{(zq)_{n_k}}=\dfrac{1}{(zq)_{n_0}}.
\label{eqn:andrewswarnaar}
\end{align}
Letting $k\rightarrow\infty$ (i.e., iterating the relation infinitely many times) 
we see:
\begin{align}
\sum_{n_1,n_2,\cdots=0}^{n_0}
\qbin{n_0}{n_1}
\left(
\prod_{i\geq 1} z^{n_i}q^{n_i^2}
\qbin{n_i}{n_{i+1}}
\right)
=
\sum_{n_1,n_2,\cdots=0}^{n_0}
\left(
\prod_{i\geq 1} z^{n_i}q^{n_i^2}
\qbin{n_{i-1}}{n_{i}}
\right)
=\dfrac{1}{(zq)_{n_0}}.
\label{eqn:andrewswarnaar_inf_fin}
\end{align}
Now with $n_0\rightarrow\infty$, 
we immediately get \eqref{eqn:andrewswarnaar_inf}.

Now we prove \eqref{eqn:InfiniteAG} for all $t$. Note that setting $z=1$ in \eqref{eqn:andrewswarnaar_inf} recovers the $t=\infty$ case of 
\eqref{eqn:InfiniteAG} and 
setting $z=q$ recovers the $t=1$ case.
We start with an easy lemma.
\begin{lem}
\label{lem:bincoeff}
We have:
\begin{align}
\sum_{k=0}^{n}
\dfrac{z^{k}q^{k^2}}{(zq^2;q)_{k}}
\qbin{n}{k}
&= \dfrac{1+zq-zq^{n+1}}{(zq^2)_{n}}=\dfrac{1 -zq^{n+1}-z^2 q^2 (1-q^{n})}{(zq;q)_{n+1}}
\end{align}
\end{lem}
\begin{proof}
Denote the LHS as $f_{n}(z,q)$.
We have $f_0(z,q)=1$ and
\begin{align}
f_{n}(z,q)-&f_{n-1}(z,q)=
\sum_{k=0}^{n}
\dfrac{z^{k}q^{k^2}}{(zq^2;q)_{k}}
\qbin{n}{k}
-
\sum_{k=0}^{n-1}
\dfrac{z^{k}q^{k^2}}{(zq^2;q)_{k}}
\qbin{n-1}{k}
\notag\\
&=
\sum_{k=1}^{n}
\dfrac{z^{k}q^{k^2}}{(zq^2;q)_{k}}
\left(\qbin{n}{k}-\qbin{n-1}{k}\right)
=
q^n\sum_{k=1}^{n}
\dfrac{z^{k}q^{k^2-k}}{(zq^2;q)_{k}}
\qbin{n-1}{k-1}
\notag\\
&
=
zq^n\sum_{k=0}^{n-1}
\dfrac{z^{k}q^{k^2+k}}{(zq^2;q)_{k+1}}
\qbin{n-1}{k}
=
\dfrac{zq^n}{1-zq^2}\sum_{k=0}^{n-1}
\dfrac{z^{k}q^{k^2+k}}{(zq^3;q)_{k}}
\qbin{n-1}{k}
=
\dfrac{zq^n}{1-zq^2}f_{n-1}(zq,q).
\end{align}
It is straightforward to check that the RHS satisfies the same initial condition and recurrence.
\end{proof}

S.\ Ole Warnaar kindly let us know that the lemma above can be generalized.

\begin{lem}[S.\ Ole Warnaar \cite{War-email}]
\label{lem:war_gen}
For $m\geq 0$ we have:
\begin{align}
\sum_{k=0}^n \frac{z^k q^{k^2}}{(zq;q)_{m+k}} \qbin{n}{k}
=\sum_{k=0}^m \frac{z^k q^{k^2}}{(zq;q)_{n+k}} \qbin{m}{k}
    = \frac{1}{(zq;q)_{n+m}}
       \sum_{i=0}^{m}(-z)^i q^{in+\binom{i+1}{2}} \qbin{m}{i}
            \sum_{k=0}^{m-i} z^k q^{k(k+i)} \qbin{m-i}{k}. 
\label{eqn:war_gen}            
\end{align}            
\end{lem}
\begin{proof}
Note that taking $m=1$ recovers Lemma \ref{lem:bincoeff}.

Denote the LHS of \eqref{eqn:war_gen} by $f_{n,m}(z,q)$. 
It is clear that 
\begin{align}
f_{0,m}(z,q)=\dfrac{1}{(zq;\,q)_m}.
\label{eqn:war_gen_n=0}
\end{align}
Also by \eqref{eqn:basicbin},
\begin{align}
f_{n,0}(z,q)=\sum_{k=0}^n\frac{z^k q^{k^2}}{(zq;\,q)_{k}} \qbin{n}{k}=
\dfrac{1}{(zq;\,q)_n}.
\label{eqn:war_gen_m=0}
\end{align}
Proceeding as in the proof of Lemma \eqref{lem:bincoeff}, we see:
\begin{align}
f_{n,m}(z,q)-f_{n-1,m}(z,q)=\dfrac{zq^n}{1-zq}f_{n-1,m}(zq,q).
\label{eqn:war_gen_recn}
\end{align}
Also,
\begin{align}
f_{n,m}(z,q)&-f_{n,m-1}(z,q)
=\sum_{k=0}^n z^kq^{k^2}\qbin{n}{k}\left(\dfrac{1}{(zq;\,q)_{m+k}}-\dfrac{1}{(zq;\,q)_{m-1+k}}\right)
=zq^{m}\sum_{k=0}^n \dfrac{z^kq^{k^2+k}}{(zq;\,q)_{m+k}}\qbin{n}{k}\notag\\
&=\dfrac{zq^{m}}{1-zq}\sum_{k=0}^n \dfrac{z^kq^{k^2+k}}{(zq^2;\,q)_{m+k-1}}\qbin{n}{k}
=\dfrac{zq^{m}}{1-zq}f_{n,m-1}(zq,q).
\label{eqn:war_gen_recm}
\end{align}
Since the initial conditions \eqref{eqn:war_gen_n=0}, \eqref{eqn:war_gen_m=0} and the recurrences \eqref{eqn:war_gen_recn}, \eqref{eqn:war_gen_recm} are symmetric in $n$ and $m$, $f_{n,m}=f_{m,n}$. This gives the first equality.
For the rest, we proceed as follows.
\begin{align}
&
\sum_{i=0}^{m}(-z)^i q^{in+\binom{i+1}{2}} \qbin{m}{i}
	\sum_{k=0}^{m-i} z^k q^{k(k+i)} \qbin{m-i}{k}
=
\sum_{k=0}^{m}{z^kq^{k^2}}\qbin{m}{k}
	\sum_{i=0}^{m-k}{(-z)^iq^{i(k+n)+\binom{i+1}{2}}}
    \qbin{m-k}{i}\notag\\
&=
\sum_{k=0}^{m}{z^kq^{k^2}}\qbin{m}{k}
   	(zq^{k+n+1};\,q)_{m-k}
=
\sum_{k=0}^{m}{z^kq^{k^2}}\qbin{m}{k}
   	\dfrac{(zq;\,q)_{m+n}}{(zq;\,q)_{k+n}},
\end{align}
where the second step is due to the $q$-binomial theorem.
This gives the second equality of \eqref{eqn:war_gen}.
\end{proof}

We now have the following identity, which generalizes \eqref{eqn:andrewswarnaar_inf_fin}.
We thank S.\ Ole Warnaar for pointing out that it can also be proved using Hall--Littlewood
polynomials.
\begin{thm}[cf.\ {\cite[p.\ 412]{War-HL}}]
Let $t\in\ZZ_{>0}\cup\{\infty\}$. Then,
\begin{align}
\sum_{n_1,n_2,\cdots=0}^{n_0}
\left(\prod_{1\leq i < t}z^{n_i}q^{n_i^2}
\qbin{n_{i-1}}{n_{i}}
\prod_{j \geq t}z^{n_j}q^{n_j^2+n_j}
\qbin{n_{j-1}}{n_{j}}\right)
= \dfrac{1 -zq^{n_0+1}-z^t q^t (1-q^{n_0})}{(zq;q)_{n_0+1}}.
\label{eqn:truncinfiniteAG}
\end{align}
\end{thm}
\begin{proof}
Note that \eqref{eqn:truncinfiniteAG} with $t=\infty$ corresponds to \eqref{eqn:andrewswarnaar_inf_fin} and 
\eqref{eqn:truncinfiniteAG} with 
$t=1$
corresponds to \eqref{eqn:andrewswarnaar_inf_fin} with $z\mapsto zq$.
To prove \eqref{eqn:truncinfiniteAG} for $2\leq t<\infty$, we induct on $t$.
First, let $t=2$. We have, using \eqref{eqn:andrewswarnaar_inf_fin} with $z\mapsto zq$, 
\begin{align}
\sum_{n_1,n_2,\cdots=0}^{n_0}
&z^{n_1}q^{n_1^2}
\qbin{n_{0}}{n_{1}}
\prod_{j \geq 2}z^{n_j}q^{n_j^2+n_j}
\qbin{n_{j-1}}{n_{j}}
=\sum_{n_1=0}^{n_0}z^{n_1}q^{n_1^2}
\qbin{n_{0}}{n_{1}}
\left(
\sum_{n_2,n_3\cdots=0}^{n_1}
\prod_{j \geq 2}z^{n_j}q^{n_j^2+n_j}
\qbin{n_{j-1}}{n_{j}}
\right)
\nonumber\\
&=\sum_{n_1=0}^{n_0}
\dfrac{z^{n_1}q^{n_1^2}}{(zq^2;q)_{n_1}}
\qbin{n_{0}}{n_{1}}
=\dfrac{1 -zq^{n_0+1}-z^2 q^2 (1-q^{n_0})}{(zq;q)_{n_0+1}}.
\label{eqn:l=2}
\end{align}
Now, for $t>2$, we consider:
\begin{align}
\sum_{n_1,n_2,\cdots=0}^{n_0}
&\prod_{1\leq i < t+1} z^{n_{i}}q^{n_i^2}
\qbin{n_{i-1}}{n_{i}}
\prod_{j \geq t+1}z^{n_j}q^{n_j^2+n_j}
\qbin{n_{j-1}}{n_{j}}\nonumber\\
&=
\sum_{n_1=0}^{n_0}
z^{n_{1}}q^{n_1^2}
\qbin{n_0}{n_1}
\left(
\sum_{n_2,n_3,\cdots=0}^{n_1}
\prod_{2\leq i < t+1} z^{n_{i}}q^{n_i^2}
\qbin{n_{i-1}}{n_{i}}
\prod_{j \geq t+1}z^{n_j}q^{n_j^2+n_j}
\qbin{n_{j-1}}{n_{j}}
\right)
\nonumber\\
&=\sum_{n_1=0}^{n_0}
z^{n_{1}}q^{n_1^2}
\qbin{n_0}{n_1}
\dfrac{1 -zq^{n_1+1}-z^t q^t (1-q^{n_1})}{(zq;q)_{n_1+1}}
\nonumber\\
&=
\left(\sum_{n_1=0}^{n_0}
\dfrac{z^{n_{1}}q^{n_1^2}}{(zq)_{n_1}}
\qbin{n_0}{n_1}\right)
-\left(\dfrac{z^t q^t}{(1-zq)}
\sum_{n_1=0}^{n_0}
\dfrac{z^{n_{1}}q^{n_1^2}}{(zq^2)_{n_1}}\qbin{n_0}{n_1}\right)
+\left(\dfrac{z^t q^t}{(1-zq)}
\sum_{n_1=0}^{n_0}
\dfrac{z^{n_{1}}q^{n_1^2+n_1}}{(zq^2)_{n_1}}\qbin{n_0}{n_1}\right)
\nonumber\\
&=\dfrac{1}{(zq)_{n_0}}
-
\dfrac{z^t q^t(1+zq-zq^{n_0+1})} {(zq;q)_{n_0+1}}
+\dfrac{z^t q^t}{(zq)_{n_0+1}}
\nonumber\\
&=\dfrac{1-zq^{n_0+1} -z^{t+1}q^{t+1}(1-q^{n_0})}{(zq;q)_{n_0+1}}.
\end{align}
\end{proof}
Taking $n_0\rightarrow\infty$, we deduce a $(z,q)$--version of \eqref{eqn:InfiniteAG}:
\begin{cor}
Let $t\in\ZZ_{>0}\cup\{\infty\}$. Then,
\begin{align}
\sum_{n_1,n_2,\cdots}
\left(\prod_{1\leq i < t}z^{n_i}q^{n_i^2}
\qbin{n_{i-1}}{n_{i}}
\prod_{j \geq t}z^{n_j}q^{n_j^2+n_j}
\qbin{n_{j-1}}{n_{j}}\right)
= \dfrac{1-z^t q^t}{(zq;q)_{\infty}}.
\label{eqn:truncinfiniteAGzq}
\end{align}
\end{cor}

It is natural to ask if one can find companions to \eqref{eqn:andrewswarnaar} analogous to Andrews--Gordon identities.
Indeed, we have the following theorem which can be viewed
as a generalization of \eqref{eqn:andrewswarnaar} with $z\mapsto q$.
\begin{thm}
Let $k\geq 1$ and $1\leq t\leq k+1$. We have:
\begin{align}
\sum_{n_1,n_2,\cdots n_k=0}^{n_0}
q^{n_1^2+\cdots+n_k^2\,\,\,+n_t+n_{t+1}+\cdots+n_k}
\qbin{n_0}{n_1}\qbin{n_1}{n_2}\cdots \qbin{n_{k-1}}{n_k}
\dfrac{1}{(q^2;\,q)_{n_k}}=\dfrac{1-q^t-q^{n_0+1}(1-q^{t-1})}{(q)_{n_0+1}}.
\label{eqn:afinAG}
\end{align}
\end{thm}
\begin{proof}
We induct on $k$.
For any $k$, $t=1$ case is simply \eqref{eqn:andrewswarnaar} with $z\mapsto zq$.
Lemma \eqref{lem:bincoeff} with $z=1$ provides the case of $k=1$ and $t=2$.
So now let $k>1$ and let $1<t\leq k+1$. Call the LHS $f_{k,l}(n_0)$.
We have:
\begin{align}
f_{k,l}(n_0)&=
\sum_{n_1=0}^{n_0}
{q^{n_1^2}}\qbin{n_0}{n_1}
\left(
\sum_{n_2,n_3\cdots n_k=0}^{n_1}
q^{n_2^2+\cdots+n_k^2\,\,\,+n_t+n_{t+1}+\cdots+n_k}
\qbin{n_0}{n_1}\qbin{n_1}{n_2}\cdots \qbin{n_{k-1}}{n_k}
\dfrac{1}{(q^2;\,q)_{n_k}}
\right)\notag\\
&=\sum_{n_1=0}^{n_0}
{q^{n_1^2}}\qbin{n_0}{n_1}f_{k-1,t-1}(n_1)\\
&=\sum_{n_1=0}^{n_0}
{q^{n_1^2}}\qbin{n_0}{n_1}\dfrac{1-q^{t-1}-q^{n_1+1}(1-q^{t-2})}{(q)_{n_1+1}}
\notag\\
&=\sum_{n_1=0}^{n_0}
\dfrac{q^{n_1^2}}{(q)_{n_1}}\qbin{n_0}{n_1}
-\dfrac{q^{t-1}}{1-q}
\left(\sum_{n_1=0}^{n_0}
\dfrac{{q^{n_1^2}}}{(q^2)_{n_1}}\qbin{n_0}{n_1}
-
\sum_{n_1=0}^{n_0}
\dfrac{{q^{n_1^2+n_1}}}{(q^2)_{n_1}}\qbin{n_0}{n_1}
\right)
\\
&=\dfrac{1}{(q)_{n_0}} - \dfrac{q^{t-1}}{1-q}\left(\dfrac{1+q-q^{n+1}}{(q^2;\,q)_{n_0}}-\dfrac{1}{(q^2;\,q)_{n_0}} \right)
\notag\\
&=\dfrac{1-q^t-q^{n_0+1}(1-q^{t-1})}{(q)_{n_0+1}}.
\end{align}
Here, in the penultimate step, we use \eqref{eqn:basicbin} with $z=1$ and $z=q$ and Lemma \ref{lem:bincoeff} with $z=1$.
\end{proof}

Note that this theorem with $n_0\rightarrow\infty$ and $k\rightarrow\infty$ also gives \eqref{eqn:InfiniteAG}.

The following corollary is now an easy consequence. We will use it below 
to check that our sum-sides satisfy initial conditions as required by the Corteel--Welsh recursion.

\begin{cor}
\label{cor:initcond}
Denote the $n_0\rightarrow\infty$ limit of the LHS of \eqref{eqn:afinAG}
by $X(k,t)$.
For $1\leq t\leq k$, we have:
\begin{align}
X(k,1)=X(k,t+1)-q\cdot X(k,t)=\dfrac{1}{(q^2;\,q)_\infty}.
\end{align}
\end{cor}

\section{Infinite level \texorpdfstring{$\typeA_2^{(1)}$}{A2}}
\label{sec:infA2}

\subsection{The conjectures} 
Let us first fix the notation.
Suppose we are given two semi-infinite vectors of non-negative integers --
\begin{align}
\rho&=[\rho_1,\rho_2,\dots] \in\ZZ_{\geq 0}^\infty,\\
\sigma&=[\sigma_1,\sigma_2,\dots] \in\ZZ_{\geq 0}^\infty.
\end{align}
Define:
\begin{align}
S_\infty(\rho\mid \sigma)(z,q)=
\dfrac{1}{1-q}
\sum_{r,s\in\ZZ_{\geq 0}^\infty} z^{r_1}q^{\rho\cdot r+\sigma\cdot s}q^{\sum_{i\geq 1}r_i^2-r_is_i+s_i^2}
\prod_{i\geq 1} \dfrac{1}{(q)_{r_i-r_{i+1}}(q)_{s_i-s_{i+1}}}	
\label{eqn:Sinfty}
\end{align}
where 
\begin{align}
r&=[r_1,r_2,\dots] \in\ZZ_{\geq 0}^\infty,\\
s&=[s_1,s_2,\dots] \in\ZZ_{\geq 0}^\infty.
\end{align}
have finitely many non-zero entries.
We will explain the reason for the factor $1/(1-q)$ in Remark \ref{rem:A2infWar}.

We clearly have:
\begin{align}
S_\infty(\rho\mid\sigma)(1,q)=S_\infty(\sigma\mid\rho)(1,q).
\end{align}

We will require certain special vectors --
\begin{align}
\mathrm{e}_j = [\underbrace{0,\cdots, 0}_{j},1,1,\cdots],
\end{align}
where $j\in\ZZ_{>0}\cup\{\infty\}$.
For convenience, we will define:
\begin{align}
\mathrm{e}_{-1} = [2,1,1,1,\cdots].
\end{align}

Finally, for $a,b\in\ZZ_{\geq 0}\cup\{\infty\}$, define:
\begin{align}
H_{(\infty,a,b)}
&=
\begin{cases}
S_\infty(\mathrm{e}_a\mid \mathrm{e}_b)
- qS_\infty(\mathrm{e}_{a-1}\mid \mathrm{e}_{b-1}) & a,b>0\\
S_\infty(\mathrm{e}_a\mid \mathrm{e}_0) & b=0\\
S_\infty(\mathrm{e}_0\mid \mathrm{e}_b)
-q(1-z)S_\infty(\mathrm{e}_{-1}\mid \mathrm{e}_{b-1}) & a=0, b\neq 0.
\end{cases}
\label{eqn:Hinf}
\end{align}
Note that we have:
\begin{align}
H_{(\infty,a,b)}(1,q)=H_{(\infty,b,a)}(1,q).
\end{align}

\begin{conj}
We conjecture that for $a,b\in\ZZ_{\geq 0}\cup\{\infty\}$
\begin{align}
H_{(\infty,a,b)}(1,q)
=\dfrac{(1-q^{a+1})(1-q^{b+1})(1-q^{a+b+2})}{(q)_\infty^3}.
\label{eqn:infinitelevelA2}
\end{align}
\end{conj}

\begin{rem}
\label{rem:A2infWar}
For $(a,b)\in\{(0,0), (\infty,0), (\infty,\infty)\}$, the conjecture above holds due to results in \cite{War-A2further}.
\begin{enumerate}
	\item For $(a,b)=(0,0)$, $S_\infty(\mathrm{e}_0\mid\mathrm{e}_0)(1,q)$ is in fact the $n_0,m_0,k\rightarrow\infty$ limit of $F_{n_0,m_0;k,1,1}^{(-1)}(1,q)$
	defined in \cite[Eq.\ (7.10)]{War-A2further}. Note that for infinite $k$, \cite[Eq.\ (7.10)]{War-A2further} 
	is a sum defined over infinite
	sequences  $n_1\geq n_2\geq\dots$, $m_1\geq m_2\geq\dots$ each of which ends with an infinite number of contiguous
	zeroes.
	Due to these zeros, the factor $1/(q)_{n_k+m_k+1}$ that appears in \cite[Eq.\ (7.10)]{War-A2further} evaluates to $1/(q)_{1}$
	when $k\rightarrow\infty$. This is precisely the factor that appears outside of the RHS in \eqref{eqn:Sinfty}.
	Using \cite[Eqn.\ (7.11)]{War-A2further} and \cite[Eq.\ (8.5)]{War-A2further} with $z=q,w=q^2$, we now see that
	the required limit evaluates to $(q)^{-1}_\infty(q^2;q)_\infty^{-1}(q^3;q)_\infty^{-1}$.
	\item For $(a,b)=(\infty,0)$,
	$S_\infty(\mathrm{e}_\infty\mid\mathrm{e}_0)(1,q)$ is the $n_0,m_0,k\rightarrow\infty$ limit of $F_{n_0,m_0;k,k+1,1}^{(-1)}(1,q)$ defined in \cite[Eq.\ (7.10)]{War-A2further}.
	Using \cite[Eqn.\ (7.11)]{War-A2further} and \cite[Eq.\ (8.5)]{War-A2further} with $z=1,w=q$, this evaluates to $(q)_\infty^{-2}(q^2;q)_\infty^{-1}$,
	as required.
	\item For $(a,b)=(\infty,\infty)$, we see that
	$H_{(\infty,\infty,\infty)}(1,q)=(1-q)S_\infty(\mathrm{e}_\infty\mid\mathrm{e}_\infty)(1,q)$, which in turn can be written as
	$n_0,m_0,k\rightarrow \infty$ limit of $F^{(-1)}_{n_0,m_0,k}(1,q)$ from \cite[Eq.\ (7.1)]{War-A2further}.
	Using \cite[Prop. 7.1]{War-A2further} and \cite[Eqn.\ (8.5)]{War-A2further} with $z=w=1$, we see that the required
	limit evaluates to $(q)_\infty^{-3}$. 
\end{enumerate}

\end{rem}

\subsection{Representation-theoretic interpretation} 
In analogy to the discussion about $\typeA_1^{(1)}$ above, 
we may also think of the RHS of \eqref{eqn:infinitelevelA2} as $(q)_\infty^{-1}$ times the principal character of generalized (i.e., parabolic) Verma module for $\typeA_2^{(1)}$
at \emph{any} level $\ell$
induced from the irreducible $\mathfrak{sl}_3$ module $L(a\Lambda_1+b\Lambda_2)$.
If $a=b=\infty$, we think of it as
$(q)_\infty^{-1}$ times the principal character of \emph{any} Verma module (in the usual sense) for $\typeA_2^{(1)}$.

\section{Cylindric partitions and Corteel--Welsh recursion}
\label{sec:CWrec}
The principal characters of standard $\typeA_r^{(1)}$ modules are known to be closely related 
to generating functions of cylindric partitions, see \cite{Tin-crystal}, \cite{FodWel}. 
Here, we only review the necessary background.
For more details, see \cite{GesKra}, \cite{FodWel}, \cite{CorDouUnc}, etc.

By a composition of $\ell$, we mean an (ordered) sequence of non-negative integers $(c_0,\dots,c_r)$ such that $\ell=c_0+c_1+\cdots+c_r$.
\begin{defi}
Consider a sequence $\Lambda = (\lambda^{(0)},\dots,\lambda^{(r)})$ where each $\lambda^{(j)}$ is a partition
$\lambda^{(j)}=\lambda_1^{(j)}+\lambda_2^{(j)}+\cdots$ arranged in a weakly descending order. 
We assume that each $\lambda^{(j)}$ continues indefinitely with only finitely many non-zero entries and ends with an infinite sequence of zeros.
We say that $\Lambda$ is a cylindric partition with profile $c$ if for all $i$ and $j$,
\begin{align}
\lambda_j^{(i)}\geq \lambda_{j+c_{(i+1)}}^{(i+1)}\quad\mathrm{and}\quad \lambda_j^{(r)}\geq \lambda_{j+c_0}^{(0)}.
\end{align}
\end{defi}

Let $F_c(z,q)=F_{(c_0,c_1,\dots,c_r)}(z,q)$ be the generating function of cylindric partitions of profile $c$, where $z$ correponds
to the maximum part statistic and $q$ corresponds to the total weight of the cylindric partition.
In other words, denoting the set of cylindric partitions of profile $c$ by $\cC_c$, we have:
\begin{align}
F_c(z,q) = \sum_{\Lambda\in\cC_c} z^\mathrm{max(\Lambda)}q^{\mathrm{wt}(\Lambda)}.
\end{align}
Due to the definition of cylindric partitions, these generating functions have an obvious cyclic symmetry:
\begin{align}
F_{(c_0, c_1, \dots, c_r)}(z,q) = F_{(c_1, c_2, \dots, c_r, c_0)}(z,q).
\label{eqn:Frotsym}
\end{align}

We have the following product formula for $F_c(1,q)$, due to Borodin \cite{Bor}. However, the same formula
can be quickly deduced from a much earlier work of Gessel--Krattenthaler \cite{GesKra} as exhibited in \cite{FodWel}.
\begin{thm}
Let $c=(c_0,c_1,\dots,c_r)$ be a composition of $\ell$. Set $m=r+1+\ell$. Then we have:
\begin{align}
F_c(1,q)=
\dfrac{(q^m;q^m)_\infty^r}{(q)_\infty^{r+1}}\prod_{1\leq i<j\leq r+1}\theta(q^{j-i + c_i+c_{i+1}+\cdots+c_{j-1}};\,q^m)
\label{eqn:cylindricprod}
\end{align}
\end{thm}

We also have the following connection to representation theory.
\begin{thm}[\cite{FodWel}, \cite{War}]
Let $L(\lambda)$ be a standard module of $\typeA_r^{(1)}$ and let 
 $\chi(L(\lambda))$ denote the principally specialized character, i.e.,
\begin{align}
\chi(L(\lambda))= \left.\left(e^{-\lambda} \ch L(\lambda)\right)\right\vert_{e^{-\alpha_0},\dots,e^{-\alpha_r}\mapsto q}.
\end{align}
With $\Lambda_j$ denoting the fundamental weights of $\typeA_r^{(1)}$, consider the level $\ell$ dominant integral weight 
\begin{align}
\lambda=c_0\Lambda_0+c_1\Lambda_1+\cdots+c_r\Lambda_r.
\end{align}
Then,
\begin{align}
F_c(1,q)
&=\dfrac{1}{(q^{r+1};\,q^{r+1})_\infty}\chi(L(\lambda))
=\dfrac{1}{(q)_\infty}\dfrac{\chi(L(\lambda))}{\chi(L(\Lambda_0))}
=\dfrac{1}{(q)_\infty}{\chi(\Omega(\lambda))},
\label{eqn:cylindricrep}
\end{align}
where $\chi(\Omega(\lambda))$ is the principal character 
of the standard module $L(\lambda)$.

Moreover, when $r+1$ and $\ell$ are coprime, we have:
\begin{align}
F_c(1,q)=\dfrac{1}{(q)_\infty}\ch\left(({\mathcal{W}_{r+1}})^{r+1,\ell+r+1}_{0,\lambda}\right),
\end{align}
where in the RHS we see the appearance of $\mathcal{W}_{r+1}$ minimal model character with 
parameters $(p,p')=(r+1,\ell+r+1)$,
central charge $r\left(1-\frac{(r+2)\ell^2}{\ell+r+1}\right)$
and highest weight labelled by $(0,\lambda)$.
\end{thm}

Given all of this, we refer to $r+1$ as the rank, $\ell$ as the level of the module, $c$ as the highest weight of the standard module in question (by a slight abuse of notation), and finally $m$ as the modulus of the identities -- this
is the period of the product in \eqref{eqn:cylindricprod}.

Upon setting $z=1$,  the cyclic symmetry can be upgraded to a dihedral symmetry:
\begin{align}
F_{(c_0, c_1, c_2, \dots, c_r)}(1,q) = F_{(c_r, c_{r-1}, c_{r-2}, \dots, c_0)}(1,q).
\end{align}
The proof of this can be obtained using the product formula \eqref{eqn:cylindricprod} above, see \cite{CorDouUnc} for the case
$r=2$.
Another way to deduce it is by using \eqref{eqn:cylindricrep} and noting that 
that this symmetry corresponds to the invariance of the principal characters of standard $\typeA_r^{(1)}$
modules under the outer automorphism group, which in this case is the dihedral group of symmetries of the Dynkin
diagram of $\typeA_r^{(1)}$.

Now let us recall the Corteel--Welsh recursion that governs $F_c(z,q)$. 
Actually, it will be beneficial for us to introduce the following notations,
\begin{align}
G_c(z,q)&=(zq;\,q)_\infty F_c(z,q),\\
H_c(z,q)&=\dfrac{G_c(z,q)}{(q;\,q)_\infty}=\dfrac{(zq;\,q)_\infty}{(q;\,q)_\infty} F_c(z,q).
\end{align}
Note that we have:
\begin{align}
H_c(1,q)=F_c(1,q)=
\dfrac{(q^m;q^m)_\infty^r}{(q)_\infty^{r+1}}\prod_{1\leq i<j\leq r+1}\theta(q^{j-i + c_i+c_{i+1}+\cdots+c_{j-1}};\,q^m)
=\dfrac{1}{(q)_\infty}{\chi(\Omega(c_0\Lambda_0+\cdots +c_r \Lambda_r))}.
\label{eqn:Hcprod}
\end{align}

Now, given a composition $c$, we define the following:
\begin{align}
I_c &= \{ 0\leq i\leq r\,\vert\,c_i\neq 0\}.
\end{align}
With $c_{-1}=c_r$ and a subset $\phi\subsetneq J\subseteq I_c$, define a new composition $c(J)=(c_0(J),c_1(J),\dots,c_r(J))$ by:
\begin{align}
c_i(J) = 
\begin{cases}
c_i-1 & i\in J\,\,\mathrm{and}\,\,(i-1)\not\in J,\\
c_i+1 & i\not\in J\,\,\mathrm{and}\,\,(i-1)\in J,\\
c_i & \mathrm{otherwise}.
\end{cases}
\end{align}
It is not hard to check that if $c$ is a composition of $\ell$, then so is $c(J)$ for any $\phi\subsetneq J\subseteq I_c$.

We then have:
\begin{thm}[\cite{CorWel}]
Fix $r, \ell$.
The functions $G_c$ for all compositions $c$ of $\ell$ with length $r+1$ are the unique solutions to the following 
(finite) system of functional equations
and the initial conditions:
\begin{align}
G_c(z,q) &= \sum_{\phi\subsetneq J\subseteq I_c}(-1)^{|J|-1}(zq;\,q)_{|J|-1}G_{c(J)}(zq^{|J|};\,q),
\quad
G_c(0,q) = 1,\quad G_c(z,0)  = 1.
\end{align}
\end{thm}

\begin{cor}
Fix $r, \ell$.
The functions $H_c$ for all compositions $c$ of $\ell$ with length $r+1$ are the unique solutions to the following 
(finite) system of functional equations
and the initial conditions:
\begin{align}
H_c(z,q) &= \sum_{\phi\subsetneq J\subseteq I_c}(-1)^{|J|-1}(zq;\,q)_{|J|-1}H_{c(J)}(zq^{|J|};\,q),
\quad
H_c(0,q)  = \dfrac{1}{(q)_\infty},\quad H_c(z,0)  = 1.
\label{eqn:Hrec}
\end{align}
\end{cor}

The aim of this paper is to conjecture and in some cases prove  that certain linear combinations of sums
in Andrews--Schilling--Warnaar form solve the system \eqref{eqn:Hrec} for compositions $c=(c_0,c_1,c_2)$ of length
$3$.

\section{\texorpdfstring{$\typeA_2$}{A2} identities in the Andrews--Schilling--Warnaar form}
\label{sec:conj}

In this section, our aim is to provide conjectures for $H_c(z,q)$ whenever $c=(c_0,c_1,c_2)$ is a length 3 
composition of $\ell\geq 2$. For a fixed $\ell$, we will present a specific way to arrange the corresponding
compositions (see for example \eqref{eqn:level16} for $\ell=16$). 
Based on this arrangement, we will provide conjectures for a subset of $H_c$ as certain linear combinations 
of sums in the Andrews--Schilling--Warnaar form, see Conjecture \ref{conj:main}. We will 
then prove that the remaining $H_c$ are completely deterined by the Corteel--Welsh recursions, see Theorem 
\ref{thm:remainingconj}.
Consider $k\geq 2$, let
\begin{align}
\rho &= [\rho_1,\rho_2,\dots,\rho_{k-1}]\in\ZZ^{k-1}_{\geq 0},\\
\sigma &= [\sigma_1,\sigma_2,\dots,\sigma_{k-1}]\in\ZZ^{k-1}_{\geq 0}\\
\end{align}
and define:
\begin{align}
S_{3k-1}(\rho\mid \sigma)
&= \sum_{r,s\in\ZZ^{k-1}_{\geq 0}}
\dfrac{z^{r_1}q^{\left(\sum_{1\leq i \leq k-1} r_i^2-r_is_i+s_i^2\right) + \rho\cdot r + \sigma\cdot s}}
{\prod_{1\leq i\leq k-2}(q)_{r_i-r_{i+1}}(q)_{s_i-s_{i+1}}}
\cdot\dfrac{q^{2r_{k-1}s_{k-1}}}{(q)_{r_{k-1}}(q)_{s_{k-1}}(q)_{r_{k-1}+s_{k-1}+1}}
\\
S_{3k+1}(\rho\mid \sigma)
&= \sum_{r,s\in\ZZ^{k-1}_{\geq 0}}
\dfrac{z^{r_1}q^{\left(\sum_{1\leq i \leq k-1} r_i^2-r_is_i+s_i^2\right)+\rho\cdot r + \sigma\cdot s}}
{\prod_{1\leq i\leq k-2}(q)_{r_i-r_{i+1}}(q)_{s_i-s_{i+1}}}
\cdot\dfrac{1}{(q)_{r_{k-1}}(q)_{s_{k-1}}(q)_{r_{k-1}+s_{k-1}+1}}
\\
S_{3k}(\rho\mid \sigma)
&= \sum_{r,s\in\ZZ^{k-1}_{\geq 0}}
\dfrac{z^{r_1}q^{\left(\sum_{1\leq i \leq k-1} r_i^2-r_is_i+s_i^2\right)+\rho\cdot r + \sigma\cdot s}}
{\prod_{1\leq i\leq k-2}(q)_{r_i-r_{i+1}}(q)_{s_i-s_{i+1}}}
\cdot\dfrac{1}{(q)_{r_{k-1}+s_{k-1}}(q)_{r_{k-1}+s_{k-1}+1}}\qbin{r_{k-1}+s_{k-1}}{r_{k-1}}_{q^3}
\end{align}
See \cite{AndSchWar} where Andrews, Schilling and Warnaar proved sum-product identities related to some of these sums.


Define the truncated analogues of the $\mathrm{e}$ vectors accordingly:
\begin{align}
\mathrm{e}_j &= [\underbrace{0,\cdots, 0}_{j},1,1,\cdots,1] \in \ZZ^{k-1},\notag\\
\mathrm{e}_{-1} &= [2,1,1,\dots,1] \in \ZZ^{k-1}.
\label{eqn:evect}
\end{align}
Also consider:
\begin{align}
\mathrm{\delta}_j &= [\underbrace{0,\dots,0}_{j-1},1,0,\dots,0]\in \ZZ^{k-1}.
\label{eqn:dvect}
\end{align}
Note that we have:
\begin{align}
S_m(\rho\mid \sigma)(zq^j,q)=S_m(\rho+j\delta_1\mid \sigma)(z,q).
\label{eqn:Sshift}
\end{align}
In Sections \ref{sec:m567} and \ref{sec:m8910}, we will work with $k=2$ and $k=3$, respectively.
In these cases, we will write various vectors (such as $\mathrm{e}_j$) explicitly. For example,
instead of $S_9(\mathrm{e_1}\mid\mathrm{e}_2)$ we will write $S_9(0,1\mid 0,0)$, etc.

Now fix $m\in 3k+\{-1,0,1\}$ and let $\ell=m-3$. Let $c=(c_0,c_1,c_2)$ be a composition of $\ell$. 

Using the cyclic symmetry of the generating functions for cylindric partitions, we will assume that
$c_0$ is the biggest part of $c$. Arrange all compositions $(c_0,c_1,c_2)$ of $\ell$
as follows.
Let $\lambda_0+\lambda_1+\lambda_2$ be the partition corresponding to $(c_0,c_1,c_2)$, i.e., 
$(\lambda_0,\lambda_1,\lambda_2)$ is just the composition $(c_0,c_1,c_2)$ sorted in descending order.
We put the composition $(c_0,c_1,c_2)$ in row $\lambda_1$ and column $\lambda_2$. 
We shall put a horizontal line after the row $k-1$. 
As an example, consider the case of $k=6, m=19, \ell=16$:
\begin{align}
{\arraycolsep=0.3cm
\def\arraystretch{1.2}
\begin{array}{llllll}
16,0,0 		\\
15,1,0/15,0,1 & 14,1,1    \\
14,2,0/14,0,2 & 13,2,1/13,1,2  & 12,2,2\\
13,3,0/13,0,3 & 12,3,1/12,1,3  & 11,3,2/11,2,3 & 10,3,3\\
12,4,0/12,0,4 & 11,4,1/11,1,4  & 10,4,2/10,2,4 & 9,4,3/9,3,4 & 8,4,4\\
11,5,0/11,0,5 & 10,5,1/10,1,5  & 9,5,2/9,2,5   & 8,5,3/8,3,5 & 7,5,4/7,4,5 & 6,5,5\\
\hline
10,6,0/10,0,6 & 9,6,1/9,1,6    & 8,6,2/8,2,6   & 7,6,3/7,3,6 & 6,6,4\\
9,7,0/9,0,7   & 8,7,1/8,1,7    & 7,7,2\\
8,8,0
\end{array}
}
\label{eqn:level16}
\end{align}

\begin{conj}\label{conj:main}
For compositions in rows $0,1,\dots,k-1$, i.e., if we have $c_1,c_2\leq k-1$, 
we conjecture that the identities are given by truncating the infinite level identities.
\begin{align}
H_{(c_0,c_1,c_2)}
&=
\begin{cases}
S_m(\mathrm{e}_{c_1}\mid \mathrm{e}_{c_2})
- qS_m(\mathrm{e}_{c_1-1}\mid \mathrm{e}_{c_2-1}) & c_1,c_2>0\\
S_m(\mathrm{e}_{c_1}\mid \mathrm{e}_0) & c_2=0\\
S_m(\mathrm{e}_0\mid \mathrm{e}_{c_2})
-q(1-z)S_m(\mathrm{e}_{-1}\mid \mathrm{e}_{c_2-1}) & c_1=0, c_2\neq 0.
\end{cases}
\label{eqn:aboveline}
\end{align}
These are our ``seed'' or ``above the line'' conjectures.
\end{conj}

\begin{rem}
Using \cite[Prop.\ 7.3]{War}, we may put these seed conjectures into an alternate
form that is much closer to giving us manifestly non-negative sum-product identities.
\end{rem}

The following theorem is an immediate consequence of corollary \eqref{cor:initcond}.
\begin{thm}
\label{thm:initcond}
If we have $c_1,c_2\leq k-1$, then the conjectures above satisfy the 
required initial conditions of \eqref{eqn:Hrec}, i.e., $H_c(0,q)={(q)_\infty^{-1}}$.
\end{thm}

We now show that the remaining identities where $\max(c_1,c_2)\geq k$ 
are determined by Corteel--Welsh recursions.
We will do this in two steps. First, we will determine $H_c$ for $c_0\geq c_1\geq c_2\geq 0$ 
(and $\max(c_1,c_2)\geq k$).
After this, we will analyze the case $c_0\geq c_2 \geq c_1\geq 0$ (and $\max(c_1,c_2)\geq k$).
To this end, we now look at this system of recursions closely.

Let us start with the first step.

We consider three cases of the Corteel--Welsh recursion \eqref{eqn:Hrec} that are relevant to us. 
First is when the composition $(c_0,c_1,c_2)$ is sufficiently far away from the edges, i.e.,
$c_0-2\geq c_1$,  $c_1-2\geq c_2>0$. 
Let $\phi\subsetneq J\subset\{0,1,2\}=I_c$. 
Now consider the diagram \eqref{eqn:Grec_part_in}.
All such diagrams indicate the locations of $c(J)$, with labels on the arrows being the set $J$.
Our plan is to use the recursion for composition enclosed in the dashed box to determine 
the $H$ corresponding to the composition enclosed in the solid box, assuming that 
$H$s corresponding to all other compositions in the diagram are known.
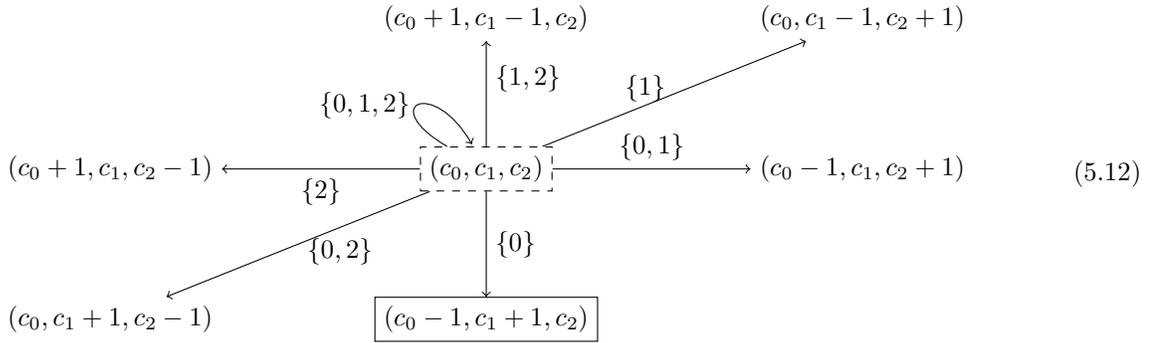
\begin{align}
\begin{matrix}
\begin{tikzpicture}
        \node[draw,dashed] (c) at (0,0) {$(c_0,c_1,c_2)$};
        \node[draw] (cj0) at (0,-2) {$(c_0-1,c_1+1,c_2)$};
        \node (cj1) at (5,2) {$(c_0,c_1-1,c_2+1)$};
        \node (cj2) at (-5,0) {$(c_0+1,c_1,c_2-1)$};
        \node (cj12) at (0,2)	{$(c_0+1,c_1-1,c_2)$};
        \node (cj01) at (5,0) {$(c_0-1,c_1,c_2+1)$};
        \node (cj02) at (-5,-2) {$(c_0,c_1+1,c_2-1)$};
        \draw[->] (c) -- (cj0) node[midway, right] {$\{0\}$};
        \draw[->] (c) -- (cj1) node[midway, left, yshift=2] {$\{1\}$};
        \draw[->] (c) -- (cj2) node[midway, below] {$\{2\}$};
        \draw[->] (c) -- (cj12) node[midway, right, yshift=6] {$\{1,2\}$};
        \draw[->] (c) -- (cj02) node[midway, right,yshift=-2.5] {$\{0,2\}$};
        \draw[->] (c) -- (cj01) node[midway, above] {$\{0,1\}$};
        \draw[->] (c) edge [out=150,in=120,loop] node[midway,above,left] {$\{0,1,2\}$} (c);
\end{tikzpicture}        
\end{matrix}
\label{eqn:Grec_part_in}
\end{align}
Because of our restriction $c_0-2\geq c_1$,  $c_1-2\geq c_2>0$, every composition 
in the diagram above is actually a partition.
Next, there are two edge cases.
If $c_0-2\geq c_1$, $c_1-2\geq c_2=0$, we have $I_c=\{0,1\}$ and then:
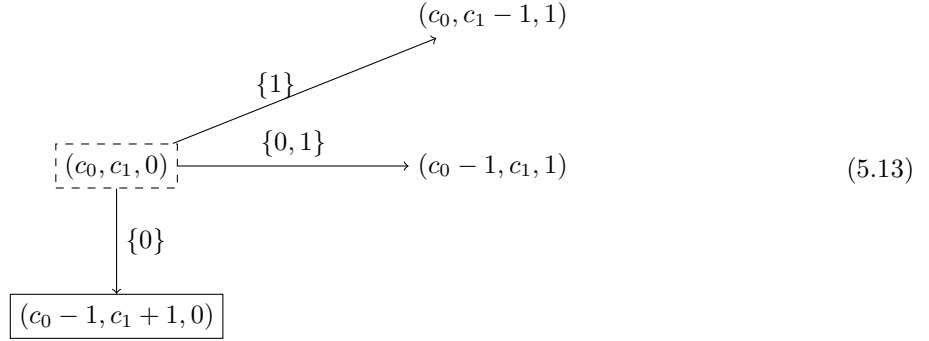
\begin{align}
\begin{matrix}
\begin{tikzpicture}
        \node[draw,dashed] (c) at (0,0) {$(c_0,c_1,0)$};
        \node[draw] (cj0) at (0,-2) {$(c_0-1,c_1+1,0)$};
        \node (cj1) at (5,2) {$(c_0,c_1-1,1)$};
        \node (cj01) at (5,0) {$(c_0-1,c_1,1)$};
        \draw[->] (c) -- (cj0) node[midway, right] {$\{0\}$};
        \draw[->] (c) -- (cj1) node[midway, left, yshift=2] {$\{1\}$};
        \draw[->] (c) -- (cj01) node[midway, above] {$\{0,1\}$};
\end{tikzpicture}        
\end{matrix}
\label{eqn:Grec_part_edge1}
\end{align}
Lastly, if $c_0-2\geq c_1$, but $c_1=c_2+1$ and $c_2>0$, we have $I_c=\{0,1,2\}$ and then:
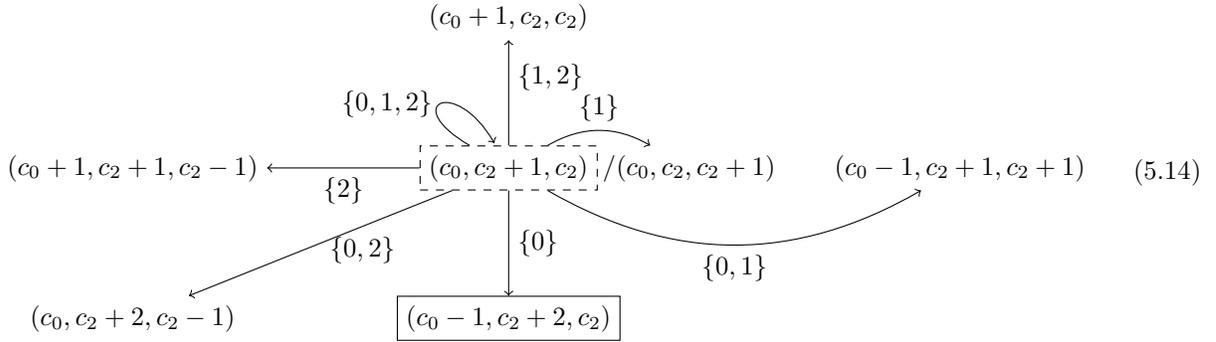
\begin{align}
\begin{matrix}
\begin{tikzpicture}
        \node[draw,dashed] (c) at (0,0) {$(c_0,c_2+1,c_2)$};
        \node (cp) at (2.4,0) {$/(c_0,c_2,c_2+1)$};
        \node[draw] (cj0) at (0,-2) {$(c_0-1,c_2+2,c_2)$};
        \node (cj2) at (-5,0) {$(c_0+1,c_2+1,c_2-1)$};
        \node (cj12) at (0,2)	{$(c_0+1,c_2,c_2)$};
        \node (cj01) at (6,0) {$(c_0-1,c_2+1,c_2+1)$};
        \node (cj02) at (-5,-2) {$(c_0,c_2+2,c_2-1)$};
        \draw[->] (c) -- (cj0) node[midway, right] {$\{0\}$};
        \draw[->] (c) to[bend left] node[midway, above] {$\{1\}$} (cp);
        \draw[->] (c) -- (cj2) node[midway, below] {$\{2\}$};
        \draw[->] (c) -- (cj12) node[midway, right, yshift=6] {$\{1,2\}$};
        \draw[->] (c) -- (cj02) node[midway, right,yshift=-2.5] {$\{0,2\}$};
        \draw[->] (c) to[bend right] node[midway, below] {$\{0,1\}$} (cj01);
        \draw[->] (c) edge [out=150,in=120,loop] node[midway,above,left] {$\{0,1,2\}$} (c);
\end{tikzpicture}        
\end{matrix}
\label{eqn:Grec_part_edge2}
\end{align}

In each of the three cases, writing the recursion for $H_{(c_0,c_1,c_2)}$ we see that the term corresponding
to $(c_0-1,c_1+1,c_2)$ is simply
$H_{(c_0-1,c_1+1,c_2)}(zq;\,q)$. This means that if expressions for all other
terms in the diagram above are known, we may deduce the expression for $H_{(c_0-1,c_1+1,c_2)}$ by using
\eqref{eqn:Sshift} with $j=-1$.

This analysis proves that if $c=(c_0,c_1,c_2)$ is a composition in lines $k,k+1,\dots$
such that $c_0\geq c_1\geq c_2\geq 0$, then $H_c$ may be recursively deduced by
using the recursion for $H_{(c_0+1,c_1-1,c_2)}$ as we move left to right in each row and then traverse
the rows from top to bottom.

In our example of $k=6, m=19, \ell=16$ above, the sequence in which we deduce $H_c$ is as follows.
Here, the arrow head points to the $H_c$ being deduced, the tail is the composition for which 
recursion is being used and the number on the arrow denotes the progression in which the deduction is performed.
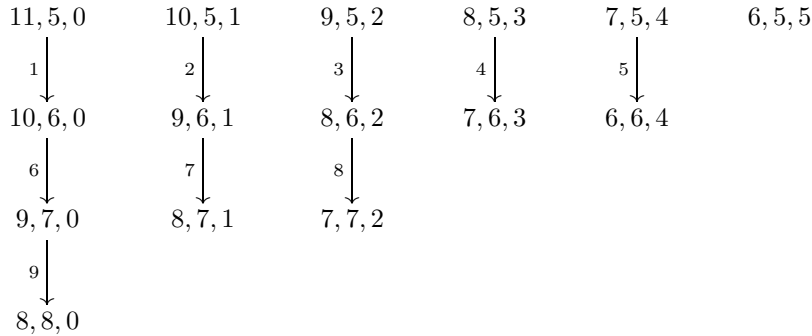
\begin{align*}
{\arraycolsep=0.35cm
\def\arraystretch{1.2}
\xymatrix{
11,5,0 \ar[d]_{1} & 10,5,1 \ar[d]_{2}  & 9,5,2 \ar[d]_{3}  & 8,5,3 \ar[d]_{4} & 7,5,4 \ar[d]_{5} & 6,5,5\\
10,6,0 \ar[d]_{6} & 9,6,1 \ar[d]_{7}   & 8,6,2\ar[d]_{8}  & 7,6,3  & 6,6,4\\
9,7,0 \ar[d]_{9}  & 8,7,1    & 7,7,2\\
{8,8,0}
}}
\end{align*}
The steps $1$, $6$, $9$ are handled by \eqref{eqn:Grec_part_edge1}, step $5$ by \eqref{eqn:Grec_part_edge2}
and rest of the steps are handled by \eqref{eqn:Grec_part_in}.

Next, we consider the case when $c=(c_0,c_1,c_2)$ with $c_0\geq c_2\geq c_1\geq 0$. Here, the analysis is similar, except that 
we move from right to left in each row $k,k+1,\dots$ and then traverse the rows top to bottom.
Remember that in this case, the composition $c$ is placed in row $c_2$ and column $c_1$.
As before, we need the following cases of the Corteel--Welsh recursions.

When the composition $(c_0,c_1,c_2)$ is sufficiently far away from the edges, i.e.,
$c_0-2\geq c_2$,  $c_2-2\geq c_1>0$. 
We have $I_c=\{0,1,2\}$, and then:
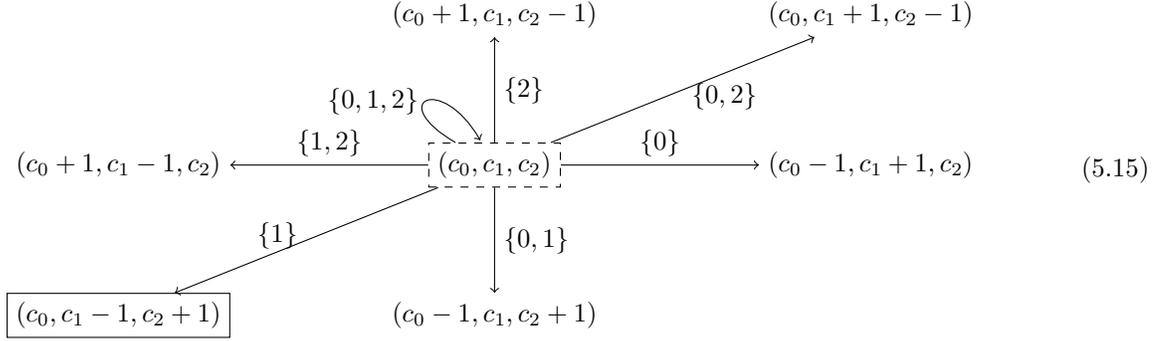
\begin{align}
\begin{matrix}
\begin{tikzpicture}
        \node[draw,dashed] (c) at (0,0) {$(c_0,c_1,c_2)$};
        \node (cj0) at (5,0) {$(c_0-1,c_1+1,c_2)$};
        \node[draw] (cj1) at (-5,-2) {$(c_0,c_1-1,c_2+1)$};
        \node (cj2) at (0,2) {$(c_0+1,c_1,c_2-1)$};
        \node (cj12) at (-5,0)	{$(c_0+1,c_1-1,c_2)$};
        \node (cj01) at (0,-2) {$(c_0-1,c_1,c_2+1)$};
        \node (cj02) at (5,2) {$(c_0,c_1+1,c_2-1)$};
        \draw[->] (c) -- (cj0) node[midway, above] {$\{0\}$};
        \draw[->] (c) -- (cj1) node[midway, left, yshift=2] {$\{1\}$};
        \draw[->] (c) -- (cj2) node[midway, right] {$\{2\}$};
        \draw[->] (c) -- (cj12) node[midway, above] {$\{1,2\}$};
        \draw[->] (c) -- (cj02) node[midway, right,yshift=-2.5] {$\{0,2\}$};
        \draw[->] (c) -- (cj01) node[midway, right] {$\{0,1\}$};
        \draw[->] (c) edge [out=150,in=120,loop] node[midway,above,left] {$\{0,1,2\}$} (c);
\end{tikzpicture}        
\end{matrix}
\label{eqn:Grec_comp_in}
\end{align}
We need one edge case of the composition $(c_0,c_0-3,c_0-2)$, with $c_0\geq 2$.
Here, we have:
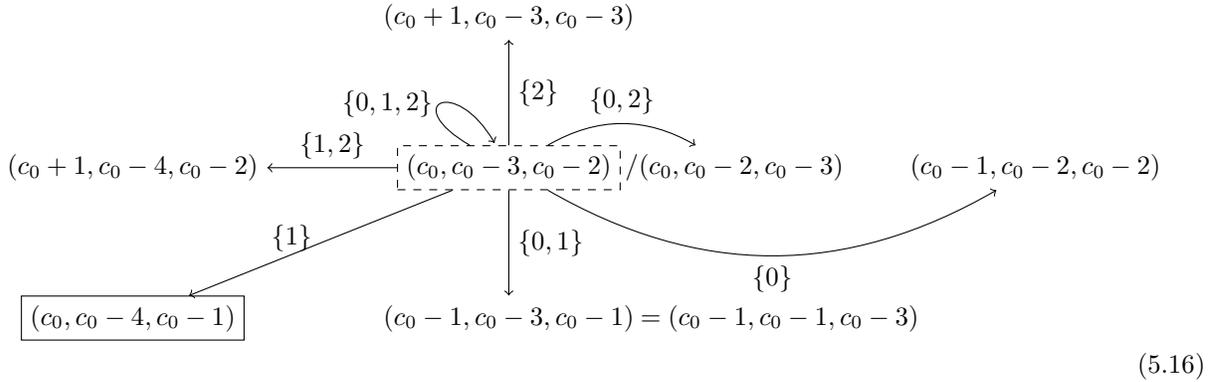
\begin{align}
\begin{matrix}
\begin{tikzpicture}
        \node[draw,dashed] (c) at (0,0) {$(c_0,c_0-3,c_0-2)$};
        \node (cp) at (3,0) {$/(c_0,c_0-2,c_0-3)$};
        \node (cj0) at (7,0) {$(c_0-1,c_0-2,c_0-2)$};
        \node[draw] (cj1) at (-5,-2) {$(c_0,c_0-4,c_0-1)$};
        \node (cj2) at (0,2) {$(c_0+1,c_0-3,c_0-3)$};
        \node (cj12) at (-5,0)	{$(c_0+1,c_0-4,c_0-2)$};
        \node (cj01) at (0,-2) {$(c_0-1,c_0-3,c_0-1)$};
        \node (cj01p) at (3.6,-2) {$=(c_0-1,c_0-1,c_0-3)$};
        \draw[->] (c) to[bend right] node[midway, below] {$\{0\}$} (cj0);
        \draw[->] (c) -- (cj1) node[midway, left, yshift=2] {$\{1\}$};
        \draw[->] (c) -- (cj2) node[midway, right] {$\{2\}$};
        \draw[->] (c) -- (cj12) node[midway, above] {$\{1,2\}$};
        \draw[->] (c) to[bend left] node[midway, above] {$\{0,2\}$} (cp) ;
        \draw[->] (c) -- (cj01) node[midway, right] {$\{0,1\}$};
        \draw[->] (c) edge [out=150,in=120,loop] node[midway,above,left] {$\{0,1,2\}$} (c);
\end{tikzpicture}        
\end{matrix}
\label{eqn:Grec_comp_edge1}
\end{align}
In these two cases, writing the recursion for $H_{(c_0,c_1,c_2)}$ we see that the term corresponding
to $(c_0,c_1-1,c_2+1)$ is simply
$H_{(c_0,c_1-1,c_2+1)}(zq;\,q)$. This means that if expressions for all other
terms in the diagram above are known, we may deduce the expression for $H_{(c_0,c_1-1,c_2+1)}$ by using
\eqref{eqn:Sshift} with $j=-1$.
Now it is not hard to see that if $c=(c_0,c_1,c_2)$ is a composition in lines $k,k+1,\dots$
such that $c_0\geq c_2\geq c_1\geq 0$, then $H_c$ may be recursively deduced by
using the recursion for $H_{(c_0,c_1+1,c_2-1)}$ as we move from right to left in each row and then traverse
the rows from top to bottom.

Again, in our example of $k=6, m=19, \ell=16$:
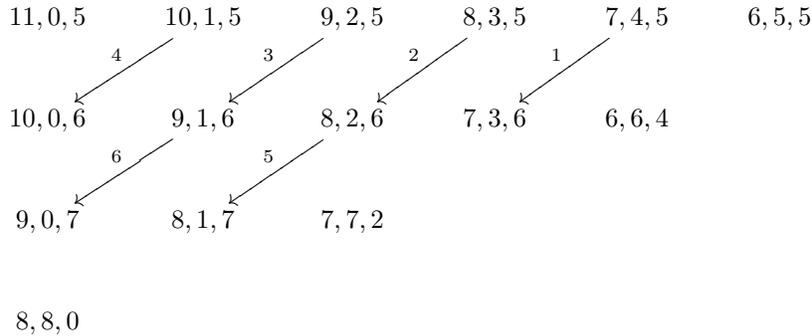
\begin{align*}
{\arraycolsep=0.35cm
\def\arraystretch{1.2}
\xymatrix{
11,0,5  & 10,1,5 \ar[dl]_{4}  & 9,2,5 \ar[dl]_{3}  & 8,3,5 \ar[dl]_{2} & 7,4,5 \ar[dl]_{1} & 6,5,5\\
10,0,6  & 9,1,6 \ar[dl]_{6}   & 8,2,6\ar[dl]_{5}  & 7,3,6  & 6,6,4\\
9,0,7   & 8,1,7  & 7,7,2\\ 
8,8,0
}}
\end{align*}
Here, step 1 is performed using \eqref{eqn:Grec_comp_edge1} and rest of the steps by \eqref{eqn:Grec_comp_in}.
  
We have thus proved that our conjecture \ref{conj:main} is enough to determine $H_c$ for all compositions
$c$ of $\ell$:

\begin{thm}\label{thm:remainingconj}
For any fixed $\ell$, identities in Conjecture \ref{conj:main} are enough to determine 
expressions for $H_c$ for all length $3$ compositions $c$ of $\ell$.
\end{thm}

\section{Moduli \texorpdfstring{$3k$}{3k} with \texorpdfstring{$z=1$}{z=1}}
In this section, we explore the mod $m=3k$ identities with $z=1$ and give precise forumlas for the 
identities below the horizontal line, i.e., in the cases when $c_1\geq k$. 
We will in fact prove these formulas also, but conditional on the validity of the seed Conjecture \ref{conj:main}
at $z=1$.
Combined with the results already present in \cite{AndSchWar}, this will help us establish a few new identities
without resorting to the Corteel--Welsh recursions. 

At $k=2$, after setting $c_0$ to be the largest part,  there are no compositions with $c_1\geq 2$. We thus
fix $k\geq 3$. Let $(c_0,c_1,c_2)$ be a composition of $\ell=3k-3$.
Since we are setting $z=1$, due to the dihedral symmetry, we may and do assume that $(c_0,c_1,c_2)$ is
a partition, i.e., $c_0\geq c_1\geq c_2\geq 0$.

For convenience, we define:
\begin{align}
\pi_{(c_0,c_1,c_2)}=\dfrac{(q^{3k},q^{3k})^2_\infty}{(q)^3_\infty}\theta(q^{c_0+1}, q^{c_1+1},q^{c_2+1};\, q^{3k})
=\dfrac{(q^{3k},q^{3k})^2_\infty}{(q)^3_\infty}\theta(q^{c_1+1},q^{c_2+1},q^{c_1+c_2+2};\, q^{3k})
,
\label{eqn:mod3kprod}
\end{align}
which is nothing but the RHS of \eqref{eqn:Hcprod}.

We have the following very useful lemma \cite[p.\ 151]{GasRah-book} originally due to Weierstra\ss:
\begin{lem}   
\label{lem:weierstrass}
Let $a_1a_2\cdots a_n=b_1b_2\cdots b_n$ and suppose that $a_i\neq q^ta_j$ for all $i\neq j$ and $t\in\ZZ$. Then:
\begin{align}
\sum_{1\leq i\leq n}\dfrac{\prod_{1\leq j\leq n}\theta(a_ib_j^{-1};\,q)}{\prod_{1\leq j\leq n, j\neq i}\theta(a_ia_j^{-1};\,q)}=0.
\end{align}
\end{lem}

Now we have the following fundamental relation for mod $3k$ identities for $k\geq 3$.
\begin{thm}
\label{thm:weier_a2_rel}
Let $k\geq 3$ and let $c_0\geq c_1\geq c_2\geq 0$ such that $c_0+c_1+c_2=3k-3$ and $c_1\geq k$.
Then,
\begin{align}
\pi_{(c_0,c_1,c_2)}=\pi_{(2k-c_2-2,2k-c_1-2,2k-c_0-2)}-q^{c_2+1}\pi_{(2k+c_2,c_0-k,c_1-k)},
\label{eqn:mod3kqrel}
\end{align}
Consequently, using \eqref{eqn:Hcprod},
\begin{align}
H_{(c_0,c_1,c_2)}(1,q)=H_{(2k-c_2-2,2k-c_1-2,2k-c_0-2)}(1,q)-q^{c_2+1}H_{(2k+c_2,c_0-k,c_1-k)}(1,q).
\label{eqn:mod3kqrelH}
\end{align}
\end{thm}
\begin{proof}
In Lemma \ref{lem:weierstrass}, take $n=3$, $q\mapsto q^{3k}$ and:
\begin{align}
\begin{matrix}
a_1=1& a_2=q^{c_1-k+1} & a_3=q^{2k-c_0-1}\\
b_1=q^k& b_2=q^{c_1+1} & b_3=q^{-c_0-1}.
\end{matrix}
\end{align}
Clearly, we have $a_1a_2a_3=b_1b_2b_3$. 
We now argue that $a_i\neq q^{3t}a_j$ for all $i\neq j$ and $t\in\ZZ$.
We have $a_1/a_2=q^{-c_1+k-1}$, $a_1/a_3=q^{c_0+1-2k}$ and $a_2/a_3=q^{c_0+c_1+2-3k}$.
Clearly, we have the following  constraints due to our restrictions on $c_0,c_1,c_2$:
\begin{align*}
2k-3\geq c_0,c_1\geq k,\quad k-3\geq c_2\geq 0
\end{align*}
which imply that none of $a_1/a_2, a_1/a_3, a_2/a_3$ could equal $q^{3kt}$ for $t\in\ZZ$.
This means that the conditions for the application of Lemma \ref{lem:weierstrass} are met.
Applying this lemma, we see:
\begin{align*}
\dfrac{\theta(q^{-k},q^{-c_1-1},q^{c_0+1};\,q^{3k})}{\theta(q^{k-c_1-1},q^{c_0+1-2k};\,q^{3k})}
+\dfrac{\theta(q^{c_1+1-2k},q^{-k},q^{c_0+c_1-k+2};\,q^{3k})}{\theta(q^{c_1-k+1},q^{c_0+c_1-3k+2};\,q^{3k})}
+\dfrac{\theta(q^{k-c_0-1},q^{2k-c_0-c_1-2},q^{2k};\,q^{3k})}{\theta(q^{2k-c_0-1},q^{3k-c_0-c_1-2};\,q^{3k})}=0.
\end{align*}
We now use the fact that $\theta(x^{-1};\,q)=-x^{-1}\theta(x;\,q)$ to modify this to:
\begin{align*}
q^{-c_0-1}&\dfrac{\theta(q^{k},q^{c_1+1},q^{c_0+1};\,q^{3k})}
{\theta(q^{-k+c_1+1},q^{-c_0-1+2k};\,q^{3k})}
-q^{-c_0-1}\dfrac{\theta(q^{-c_1-1+2k},q^{k},q^{c_0+c_1-k+2};\,q^{3k})}
{\theta(q^{c_1-k+1},q^{-c_0-c_1+3k-2};\,q^{3k})}\\
&+q^{3k-3-2c_0-c_1}\dfrac{\theta(q^{-k+c_0+1},q^{-2k+c_0+c_1+2},q^{2k};\,q^{3k})}
{\theta(q^{2k-c_0-1},q^{3k-c_0-c_1-2};\,q^{3k})}=0.
\end{align*}
Now,
observe that each of the terms has a factor $\theta(q^k;\,q^{3k})=\theta(q^{2k};\,q^{3k})$. Cancel this factor, and then 
multiply through by $q^{c_0+1}\theta(q^{c_1+1-k},q^{2k-c_0-1},q^{3k-c_0-c_1-2};\,q^{3k})$ to get:
\begin{align*}
\theta(q^{c_1+1},&q^{c_0+1},q^{3k-c_0-c_1-2};\,q^{3k})
-{\theta(q^{-c_1-1+2k},q^{c_0+c_1-k+2},q^{2k-c_0-1};\,q^{3k})}\\
&+q^{3k-2-c_0-c_1}
{\theta(q^{-k+c_0+1},q^{-2k+c_0+c_1+2},q^{c_1+1-k};\,q^{3k})}
=0.
\end{align*}
Finally, recognizing that $c_0+c_1=3k-3-c_2$ and then multiplying through
by $\dfrac{(q^{3k};\,q^{3k})^2_\infty}{(q)^3_\infty}$, we obtain the required result.
\end{proof}

The term on the LHS of \eqref{eqn:mod3kqrelH} lies below the line.
Crucially, all the terms on the RHS of \eqref{eqn:mod3kqrelH}
involve compositions $(a,b,c)$ with $a\geq b\geq c\geq 0$ and $b,c\leq k-1$,
in particular, they lie above the line. 
These terms are therefore encompassed by Conjecture \ref{conj:main}.
Substituting the corresponding expressions from Conjecture \ref{conj:main}, we obtain the following
theorem.
\begin{thm}
\label{thm:mod3krel}
Let $k\geq 3$. 
Let $c=(c_0,c_1,c_2)$ be such that $c_0\geq c_1\geq c_2\geq 0$, $c_0+c_1+c_2=3k-3$ and $c_1\geq k$.
If formulas in Conjecture \ref{conj:main} are true for $z=1$, then we must have:
\begin{align}
H_c(1,q)=
X(1,q)-q^{c_2+1}Y(1,q),
\end{align}
where:
\begin{align}
X(z,q)=
\begin{cases}
S_{3m}(\mathrm{e}_{2k-c_1-2}\mid\mathrm{e}_{2k-c_0-2})-qS_{3m}(\mathrm{e}_{2k-c_1-3}\mid\mathrm{e}_{2k-c_0-3}),
& c_0<2k-2\\
S_{3m}(\mathrm{e}_{2k-c_1-2}\mid\mathrm{e}_0),
& c_0=2k-2
\end{cases}
\end{align}
and 
\begin{align}
Y(z,q)=
\begin{cases}
S_{3m}(\mathrm{e}_{c_0-k}\mid\mathrm{e}_{c_1-k})-qS_{3m}(\mathrm{e}_{c_0-k-1}\mid\mathrm{e}_{c_1-k-1}), & c_1>k\\
S_{3m}(\mathrm{e}_{c_0-k}\mid\mathrm{e}_0) & c_1=k.
\end{cases}
\end{align}
\end{thm}

Now, we observe that \cite[Thm.\ 5.4]{AndSchWar} actually proves Conjecture \ref{conj:main} with $z=1$
for the case of compositions $(3k-3-2j,j,j)$ with $0\leq j\leq k-1$. 
In our arrangement, these entries lie to the far right of rows $0,1,2,\cdots,k-1$ above the horizontal line.
We may therefore deduce new sum-product identities whenever both terms on the 
RHS of \eqref{eqn:mod3kqrelH} involve such compositions. 
This happens when $(c_0,c_1,c_2)$ in \eqref{eqn:mod3kqrelH}
is taken to be $(i,i,3k-3-2i)$ with $i\geq k$ and $3k-3-2i\geq 0$.
This is captured in the following theorem.
\begin{thm}
\label{thm:mod3k}
Let $k\geq 3$, and let $i\geq k$ such that $3k-3-2i\geq 0$.
We then have:
\begin{align}
H_{(i,i,3k-3-2i)}(1,q)
= 
X(1,q)-q^{3k-2-2i}Y(1,q),
\end{align}
where:
\begin{align}
X(z,q)=
\begin{cases}
S_{3m}(\mathrm{e}_{2k-i-2}\mid\mathrm{e}_{2k-i-2})-qS_{3m}(\mathrm{e}_{2k-i-3}\mid\mathrm{e}_{2k-i-3}),
& i<2k-2\\
S_{3m}(\mathrm{e}_{2k-i-2}\mid\mathrm{e}_0),
& i=2k-2
\end{cases}
\end{align}
and 
\begin{align}
Y(z,q)=
\begin{cases}
S_{3m}(\mathrm{e}_{i-k}\mid\mathrm{e}_{i-k})-qS_{3m}(\mathrm{e}_{i-k-1}\mid\mathrm{e}_{i-k-1}), & i>k\\
S_{3m}(\mathrm{e}_{i-k}\mid\mathrm{e}_0) & i=k.
\end{cases}
\end{align}
\end{thm}

As an example, consider the mod $9$, i.e., $k=3$ identities. 
Here, there is exactly one value of $i$, namely $i=3$ that can be used in Theorem \ref{thm:mod3k}.
Thus, corresponding to the composition $(3,3,0)$ we have proved:
\begin{align}
H_{(3,3,0)}&=\dfrac{(q^9;q^9)^2_\infty\theta(q,q^4,q^4;\,q^9)}{(q)^3_\infty}
=\dfrac{1}{(q)_\infty}\dfrac{1}{\theta(q,q^2,q^2,q^3,q^3;q^9)_\infty}
\notag\\
&=S_{9}(\mathrm{e_1}\mid \mathrm{e}_1)(1,q)-qS_9(\mathrm{e_0}\mid \mathrm{e}_0)(1,q)-qS_{9}(\mathrm{e}_0\mid\mathrm{e}_0)\notag\\
&=\sum_{r_1,r_2,s_1,s_2\geq 0}
\dfrac{q^{r_1^2-r_1s_1+s_1^2+r_2^2-r_2s_2+s_2^2+r_2+s_2}(1-2q^{1+r_1+s_1})}
{(q)_{r_1-r_2}(q)_{s_1-s_2}(q)_{r_2+s_2}(q)_{r_2+s_2+1}}
\qbin{r_2+s_2}{r_2}_{q^3}.
\label{eqn:mod9missing}
\end{align}

Experimentally, we have the following conjectural, two-variable, sum-side version of Theorem \ref{thm:weier_a2_rel}. 
\begin{conj}
We have:
\begin{align}
H_{(k,k,k-3)}(z,q)=H_{(k+1,k-2,k-2)}(z,q)-zq^{k-2}H_{(3k-3,0,0)}(z,q).
\end{align}
\end{conj}
It will be interesting to find a combinatorial proof of this.

\section{Moduli \texorpdfstring{$5$}{5}, \texorpdfstring{$6$}{6}, \texorpdfstring{$7$}{7}}
\label{sec:m567}
These moduli correspond to $k=2$. 
We state and prove all required identities for these three moduli, with the case of mod $6$ being new.

\subsection{Mod \texorpdfstring{$5$}{5}}
The arrangement is as follows.
\begin{align}
\begin{matrix}
2,0,0 \\
1,1,0\\ \hline
\end{matrix}
\end{align}

The identities are as follows. 
\begin{align}
H_{(2,0,0)} &= S_5(1\mid 1), \\
H_{(1,1,0)} &= S_5(0\mid 1). 
\end{align}

We now show that these give unique solutions to \eqref{eqn:Hrec} with $r=2$ and $\ell=2$.
The recursion for $H_{(2,0,0)}$ is satisfied immediately,
and proving the recursion for $H_{(1,1,0)}$ amounts to the following:
\begin{thm} We have
\begin{align}
zqS_5(2\mid 1) - S_5(0\mid 1)+S_5(1\mid 1)=0.	
\end{align}
\end{thm}
\begin{proof}
The coefficient of $z^0$ on the LHS is easily seen to equal $0$.
Now let $r\geq 0$.
We have:
\begin{align*}
[z^{r+1}]&\left(zqS_5(2\mid 1) - S_5(0\mid 1)+S_5(1\mid 1)\right)\\
&=
\sum_{s\geq 0}\dfrac{q^{r^2+rs+s^2+2r+s+1}}{(q)_{r}(q)_s(q)_{r+s+1}}
-\sum_{s\geq 0}\dfrac{q^{r^2+rs+s^2+2r+2s+1}}{(q)_{\mathbf{r+1}}(q)_s(q)_{r+s+2}}
+\sum_{s\geq 0}\dfrac{q^{r^2+rs+s^2+3r+2s+2}}{(q)_{\mathbf{r+1}}(q)_s(q)_{r+s+2}}
\\
&=
\sum_{s\geq 0}\dfrac{q^{r^2+rs+s^2+2r+s+1}}{(q)_{r}(q)_s(q)_{r+s+1}}
-\sum_{s\geq 0}\dfrac{q^{r^2+rs+s^2+2r+2s+1}}{(q)_{\mathbf{r}}(q)_s(q)_{r+s+2}}\\
&=\sum_{s\geq 0}\dfrac{q^{r^2+rs+s^2+2r+s+1}(1-q^{r+s+2}-q^s)}{(q)_{r}(q)_s(q)_{r+s+2}}\\
&=\sum_{s\geq 0}\dfrac{q^{r^2+rs+s^2+2r+s+1}(1-q^s)(1-q^{r+s+2})}{(q)_{r}(q)_s(q)_{r+s+2}}
-\sum_{s\geq 0}\dfrac{q^{r^2+rs+s^2+3r+3s+3}}{(q)_{r}(q)_s(q)_{r+s+2}}.
\end{align*}
Now observe:
\begin{align*}
\sum_{s\geq 0}\dfrac{q^{r^2+rs+s^2+2r+s+1}(1-q^s)(1-q^{r+s+2})}{(q)_{r}(q)_s(q)_{r+s+2}}
=\sum_{\mathbf{s\geq 1}}\dfrac{q^{r^2+rs+s^2+2r+s+1}}{(q)_{r}(q)_{{s-1}}(q)_{{r+s+1}}}
=\sum_{\mathbf{s\geq 0}}\dfrac{q^{r^2+rs+s^2+{3r+3s+3}}}{(q)_{r}(q)_{{s}}(q)_{{r+s+2}}}.
\end{align*}

\end{proof}
Combined with Theorem \ref{thm:initcond}, we have established Conjecture \ref{conj:main}
 compositions of $2$.

\subsection{Mod \texorpdfstring{$6$}{6}}

The arrangement is as follows.
\begin{align}
\begin{array}{ll}
3,0,0 & \\
2,1,0/2,0,1  & 1,1,1\\
\hline
\end{array}
\end{align}
The identities are:
\begin{align}
H_{(3,0,0)} &= S_6(1\mid 1)\\
H_{(2,1,0)} &= S_6(0\mid 1)\\
H_{(2,0,1)} &= S_6(1\mid 0)-q(1-z)S_6(2\mid 1)\\
H_{(1,1,1)} &= S_6(0\mid 0)-qS_6(1\mid 1).
\end{align}

We now show that these give unique solutions to \eqref{eqn:Hrec} with $r=2$ and $\ell=3$.
The recurrences for $H_{(3,0,0)}$, $H_{(2,0,1)}$ are satisfied immediately.
Before proving the remaining recurrences, we first find and prove some
relations satisfied by the sums.

\begin{lem} For $A\in\ZZ$, we have:
\begin{align}
S_6(A \mid 0) - (1+q)S_6(1+A \mid 1)   -zq^{1+A}S_6(2+A\mid -1) = 0
\tag{\mbox{$R_1$}}
\label{eqn:m6r1}
\end{align}
\end{lem}
\begin{proof}
It is enough to prove the case for $A=0$. The $A\neq 0$ case can be obtained from this by
considering $z\mapsto zq^A$.
First consider the constant term (w.r.t.\ $z$) in the LHS.
This is --
\begin{align*}
\sum_{s\geq 0} &\dfrac{q^{s^2}(1-q^s-q^{s+1})}{(q)_{s}(q)_{s+1}}
=\sum_{s\geq 0} \dfrac{q^{s^2}}{(q)^2_{s}}-\sum_{s\geq 0} \dfrac{q^{s^2+s}}{(q)_{s}(q)_{s+1}}=0,
\end{align*}
where the last equality follows since both the sums equal $1/(q)_{\infty}$; first due to Durfee squares,
second due to Durfee rectangles.
Now, with $r\geq 0$, we have:
\begin{align*}
[z^{r+1}]&\left(S_6(0 \mid 0) - (1+q)S_6(1 \mid 1)   -zqS_6(2\mid -1)\right)\\
&=\sum_{s\geq 0}\dfrac{q^{r^2-rs+s^2+2r-s+1}(1-q^{r+s+1}-q^{r+s+2})}{(q)_{r+s+1}(q)_{r+s+2}}\qbin{r+1+s}{r+1}_{q^3}
-\sum_{s\geq 0}\dfrac{q^{r^2-rs+s^2+2r-s+1}}{(q)_{r+s}(q)_{r+s+1}}\qbin{r+s}{r}_{q^3}\\
&=\sum_{s\geq 0}\dfrac{q^{r^2-rs+s^2+2r-s+1}}{(q)_{r+s+1}(q)_{r+s+2}(1-q^{3r+3})}\qbin{r+s}{r}_{q^3}\\
&\quad\quad\quad\times  \left((1-q^{3r+3s+3})(1-q^{r+s+1}-q^{r+s+2})-(1-q^{r+s+1})(1-q^{r+s+2})(1-q^{3r+3})\right)\\
&=\sum_{s\geq 0}\dfrac{q^{r^2-rs+s^2+2r-s+1}}{(q)_{r+s+1}(q)_{r+s+2}(1-q^{3r+3})}\qbin{r+s}{r}_{q^3}\\
&\quad\quad\quad\times  \left(
q^{3r+3}(1-q^{3s})(1-q^{r+s+1})(1-q^{r+s+2}) - q^{2r+2s+3}(1-q^{3r+3s+3})
	\right)\\
&=	\sum_{\mathbf{s\geq 1}}\dfrac{q^{r^2-rs+s^2+5r-s+4}(q^3;q^3)_{r+s}}{(q)_{r+s}(q)_{r+s+1}(q^3;q^3)_{r+1}(q^3;q^3)_{s-1}}
-\sum_{s\geq 0}\dfrac{q^{r^2-rs+s^2+4r+s+4}}{(q)_{r+s+1}(q)_{r+s+2}}\qbin{r+s+1}{r+1}_{q^3}\\
&=	\sum_{\mathbf{s\geq 0}}\dfrac{q^{r^2-rs+s^2+4r+s+4}}{(q)_{r+s+1}(q)_{r+s+2}}\qbin{r+s+1}{r}_{q^3}
-\sum_{s\geq 0}\dfrac{q^{r^2-rs+s^2+4r+s+4}}{(q)_{r+s+1}(q)_{r+s+2}}\qbin{r+s+1}{r}_{q^3}\\
&=0.
\end{align*}

\end{proof}

\begin{lem}
We have:
\begin{align}
2S_6(1\mid 1) -S_6(2\mid -1)+q(1-zq)S_6(3\mid 0)  = 0
\tag{\mbox{$R_2$}}
\label{eqn:m6r2}
\end{align}
\end{lem}
\begin{proof}
To check that the constant term with respect to $z$ is zero, we need the following identity:
\begin{align*}
\sum_{s\geq 0}\dfrac{q^{s^2}(2q^s-q^{-s}+q)}{(q)_{s}(q)_{s+1}}=0.
\end{align*}
For this, note that
\begin{align*}
\sum_{s\geq 0}\dfrac{2q^{s^2+s}}{(q)_{s}(q)_{s+1}}=\dfrac{2}{(q)_\infty},
\end{align*}
and,
\begin{align*}
\left(\sum_{s\geq 0}\right.&\left.\dfrac{q^{s^2}(q^{-s}-q)}{(q)_{s}(q)_{s+1}}\right)-\frac{1}{(q)_\infty}=
\left(\sum_{s\geq 0}\dfrac{q^{s^2-s}}{(q)_{s}^2}\right) - \frac{1}{(q)_\infty}
=\sum_{s\geq 0}\dfrac{q^{s^2-s}}{(q)_{s}^2} - \sum_{s\geq  0}\dfrac{q^{s^2}}{(q)_s^2}\\
&=\sum_{s\geq 0}\dfrac{q^{s^2-s}(1-q^s)}{(q)_{s}^2}
=\sum_{s\geq 1}\dfrac{q^{s^2-s}}{(q)_{s-1}(q)_s}
=\sum_{s\geq 0}\dfrac{q^{s^2+s}}{(q)_{s}(q)_{s+1}}
=\dfrac{1}{(q)_\infty},
\end{align*}
as required.

Now, with $r\geq 0$, we have --
\begin{align*}
&[z^{r+1}]\left(2S_6(1\mid 1) -S_6(2\mid -1)+q(1-zq)S_6(3\mid 0)\right)\\
&=\sum_{s\geq 0}
\dfrac{q^{r^2-rs+s^2}(2q^{3r+2} -q^{4r-2s+3}+q^{5r-s+5})}{(q)_{r+s+1}(q)_{r+s+2}}\qbin{r+s+1}{r+1}_{q^3}
-
\sum_{s\geq 0}
\dfrac{q^{r^2-rs+s^2+3r+2}}{(q)_{r+s}(q)_{r+s+1}}\qbin{r+s}{r}_{q^3}\\
&=
\sum_{s\geq 0}
\dfrac{q^{r^2-rs+s^2}}{(q)_{r+s+1}(q)_{r+s+2}(1-q^{3r+3})}\qbin{r+s}{r}_{q^3}\\
&\quad\times\left( (1-q^{3r+3s+3})(2q^{3r+2} -q^{4r-2s+3}+q^{5r-s+5})-q^{3r+2}(1-q^{r+s+1})(1-q^{r+s+2})(1-q^{3r+3})\right)\\
&=
\sum_{s\geq 0}
\dfrac{q^{r^2-rs+s^2}}{(q)_{r+s+1}(q)_{r+s+2}(1-q^{3r+3})}\qbin{r+s}{r}_{q^3}\\
&\quad\times\left( 
(q^{4r-2s+3}+q^{5r-s+4})(1-q^{r+s+1})(1-q^{r+s+2})(1-q^{3s})-(q^{3r+2}+q^{4r+s+4})(1-q^{3r+3s+3})
\right).
\end{align*}
We then have:
\begin{align*}
&\sum_{s\geq 0}
\dfrac{q^{r^2-rs+s^2}(q^{4r-2s+3}+q^{5r-s+4})(1-q^{r+s+1})(1-q^{r+s+2})(1-q^{3s})}{(q)_{r+s+1}(q)_{r+s+2}(1-q^{3r+3})}\qbin{r+s}{r}_{q^3}\\
&=
\sum_{\mathbf{s\geq 1}}
\dfrac{q^{r^2-rs+s^2}(q^{4r-2s+3}+q^{5r-s+4})(q^3;q^3)_{r+s}}{(q)_{r+s}(q)_{r+s+1}(q^3;q^3)_{r+1}(q^3;q^3)_{s-1}}
=
\sum_{\mathbf{s\geq 0}}
\dfrac{q^{r^2-rs+s^2}(q^{3r+2}+q^{4r+s+4})(1-q^{3r+3s+3})}{(q)_{r+s+1}(q)_{r+s+2}(1-q^{3r+3})}\qbin{r+s}{r}_{q^3}.
\end{align*}

\end{proof} 
We finally need the following relation.
\begin{lem} For $A\in\ZZ$ we have:
\begin{align}
S_6(A\mid 0)&-qS_6(A+1\mid 1)-S_6(A+3\mid 0)+qS_6(A+4\mid 1)\notag\\
&-(zq^{1+A}S_6(A+2,-1)+zq^{2+A}S_6(A+3\mid 0)+zq^{3+A}S_6(A+4\mid 1))=0.
\tag{\mbox{$R_3$}}
\label{eqn:m6r3}
\end{align}
\begin{proof}
The proof of the $A=0$ case is as follows.
\begin{align*}
S_6&(0\mid 0)-qS_6(1\mid 1)-S_6(3\mid 0)+qS_6(4\mid 1)\\
&=
\sum_{r,s\geq 0}
z^{r}\dfrac{q^{r^2-rs+s^2}(1-q^{r+s+1}-q^{3r}+q^{4r+s+1})}{(q)_{r+s}(q)_{r+s+1}}\qbin{r+s}{r}_{q^3}
\\
&
=\sum_{r,s\geq 0}
z^{r}\dfrac{q^{r^2-rs+s^2}(1-q^{r+s+1})(1-q^{3r})}{(q)_{r+s}(q)_{r+s+1}}\qbin{r+s}{r}_{q^3}
\\
&
=\sum_{\mathbf{r\geq 1},s\geq 0}
z^{r}\dfrac{q^{r^2-rs+s^2}(q^3;q^3)_{r+s}}{(q)_{r+s}(q)_{r+s}(q^3;q^3)_{s}(q^3;q^3)_{r-1}}
\\
&
=\sum_{\mathbf{r\geq 0},s\geq 0}
z^{r+1}\dfrac{q^{r^2-rs+s^2+2r-s+1}(1-q^{3r+3s+3})}{(1-q^{r+s+1})(q)_{r+s}(q)_{r+s+1}}\qbin{r+s}{r}_{q^3}
\\
&
=\sum_{r,s\geq 0}
z^{r+1}\dfrac{q^{r^2-rs+s^2+2r-s+1}(1+q^{r+s+1}+q^{2r+2s+2})}{(q)_{r+s}(q)_{r+s+1}}\qbin{r+s}{r}_{q^3}
\\
&
=zqS_6(2\mid -1)+zq^{2}S_6(3\mid 0)+zq^{3}S_6(4\mid 1).
\end{align*}
\end{proof}
\end{lem}

These relations now prove the remaining recurrences.

Recurrence for $H_{(1,1,1)}$ amounts to showing:
\begin{align}
-S_6(0\mid 0)+(3+q)S_6(1\mid 1)-(2+zq^2)(1-zq)S_6(3\mid 0)+2q(1-zq^2)(1-zq)S_6(4\mid 1)=0.
\label{eqn:rec111}
\end{align}
It is easy to see that
\eqref{eqn:rec111} can be obtained as:
\begin{align*}
(2zq-3)\ref{eqn:m6r1}(0)  +zq\ref{eqn:m6r2} +2(1-zq)\ref{eqn:m6r3}(0).
\end{align*}

Recurrence for $H_{(2,1,0)}$ amounts to:
\begin{align}
-S_6(0\mid 1)+(1+zq)S_6(2\mid 0)-q(1-zq)S_6(3\mid 1)=0,
\end{align}
and
can be obtained as:
\begin{align*}
\ref{eqn:m6r1}(-1)  -\ref{eqn:m6r3}(-1).
\end{align*}

Combined with Theorem \ref{thm:initcond}, we have established Conjecture \ref{conj:main}
for compositions of $3$.

Setting $z\mapsto 1$ and using \eqref{eqn:Hcprod}, we arrive at the following sum-product identities.
\begin{cor}
We have:
\begin{align}
\sum_{r,s\geq 0}
&
\dfrac{q^{r^2-rs+s^2+r+s}}{(q)_{r+s}(q)_{r+s+1}}\qbin{r+s}{r}_{q^3}
=
\dfrac{1}{(q)_\infty}\dfrac{1}{\theta(q^2,q^3;\,q^6)}
=\dfrac{\chi(\Omega(3\Lambda_0))}{(q)_\infty}
\label{eqn:300},
\\
\sum_{r,s\geq 0}
&
\dfrac{q^{r^2-rs+s^2+s}}{(q)_{r+s}(q)_{r+s+1}}\qbin{r+s}{r}_{q^3}
=
\dfrac{1}{(q)_\infty}\dfrac{1}{\theta(q,q^2;\,q^6)}
=\dfrac{\chi(\Omega(2\Lambda_0+\Lambda_1))}{(q)_\infty}
\label{eqn:210},
\\
\sum_{r,s\geq 0}
&
\dfrac{q^{r^2-rs+s^2}(1-q^{1+r+s})}{(q)_{r+s}(q)_{r+s+1}}\qbin{r+s}{r}_{q^3}
=
\dfrac{1}{(q)_\infty}\dfrac{\theta(q^2;\,q^6)}{\theta(q,q,q^3;\,q^6)}
=\dfrac{\chi(\Omega(\Lambda_0+\Lambda_1+\Lambda_2))}{(q)_\infty}
\label{eqn:111}.
\end{align}
\end{cor}
Identities \eqref{eqn:300} and \eqref{eqn:111} were established in 
\cite{AndSchWar}.

\subsection{Mod 7} The arrangement is as follows.
\begin{align}
\begin{array}{ll}
4,0,0 & \\
3,1,0/3,0,1  & 2,1,1 \\
\hline
2,2,0 & 
\end{array}
\end{align}
With this arrangement, the identities are:
\begin{align}
H_{(4,0,0)} &= S_7(1\mid 1)\\
H_{(3,1,0)} &= S_7(0\mid 1)\\
H_{(3,0,1)} &= S_7(1\mid 0)-q(1-z)S_7(2\mid 1)\\
H_{(2,1,1)} &= S_7(0\mid 0)-qS_7(1\mid 1)\\
H_{(2,2,0)} &= S_7(-1\mid 1)-zS_7(1 \mid 0).
\end{align} 
Here, the expression for $H_{(2,2,0)}$ has been obtained by solving the recurrence
for $H_{(3,1,0)}$.
We now show that these give unique solutions to \eqref{eqn:Hrec} with $r=2$ and $\ell=4$.
Recurrences for $H_{(4,0,0)}$, $H_{(3,1,0)}$ and $H_{(3,0,1)}$ are seen to be 
satisfied immediately.
To get the remaining recurrences, we prove the following relations.
\begin{lem} For $A\in\ZZ$ we have:
\begin{align}
	S_7(A\mid 0)-S_7(A+1\mid 0)-qS_7(A+1\mid 1)+qS_7(A+2\mid 1)-zq^{A+1}S_7(A+2\mid -1)=0
	\tag{\mbox{$R_1$}}
	\label{eqn:m7r1}
\end{align}
\end{lem}
\begin{proof}
It is enough to check the $A=0$ case. We have:
\begin{align*}
S_7&(0\mid 0)-S_7(1\mid 0)-qS_7(1\mid 1)+qS_7(2\mid 1)\\
&=\sum_{r,s\geq0}z^r\dfrac{q^{r^2-rs+s^2}(1-q^r-q^{r+s+1}+q^{2r+s+1})}{(q)_{r}(q)_s(q)_{r+s+1}}
=\sum_{r,s\geq0}z^r\dfrac{q^{r^2-rs+s^2}(1-q^r)(1-q^{r+s+1})}{(q)_{r}(q)_s(q)_{r+s+1}}\\
&=\sum_{\mathbf{r\geq 1},s\geq0}z^r\dfrac{q^{r^2-rs+s^2}}{(q)_{r-1}(q)_s(q)_{r+s}}
=\sum_{\mathbf{r\geq 0},s\geq0}z^{r+1}\dfrac{q^{r^2-rs+s^2+2r-s+1}}{(q)_{r}(q)_s(q)_{r+s+1}}=zqS_7(2\mid -1).
\end{align*}
\end{proof}
\begin{rem}
We may also use the $(1-q^s)$ factor to deduce a similar relation.
We will actually deduce a slight variant of this relation in the process proving the following lemma.
\end{rem}

\begin{lem}
For $A\in\ZZ$ we have:
\begin{align}
S_7(A\mid 1)+S_7(A+1\mid 0)-S_7(A+1\mid -1)+qS_7(A+2\mid 0)-qS_7(A+2\mid 1)=0
\tag{\mbox{$R_2$}}
\label{eqn:m7r2}
\end{align}
\end{lem}
\begin{proof}
Let $A=0$. Consider the coefficient of $z^r$ with $r\geq 0$.
We have:
\begin{align*}
[z^{r}]\left(S_7(0\mid 1)+S_7(1\mid 0)-S_7(1\mid -1)\right)
=\sum_{s\geq 0}\dfrac{q^{r^2-rs+s^2}(q^r+q^s-q^{r-s})}{(q)_r(q)_s(q)_{r+s+1}},
\end{align*}
and
\begin{align*}
[z^{r}]\left(qS_7(2\mid 0)-qS_7(2\mid 1)\right)&=\sum_{s\geq 0}\frac{q^{r^2-rs+s^2+1}(1-q^s)}{(q)_r(q)_s(q)_{r+s+1}}
=\sum_{\mathbf{s\geq 1}}\frac{q^{r^2-rs+s^2+1}}{(q)_r(q)_{s-1}(q)_{r+s+1}}
=\sum_{\mathbf{s\geq 0}}\frac{q^{r^2-rs+s^2-r+2s+2}}{(q)_r(q)_{s}(q)_{r+s+2}}.
\end{align*}
Adding the two, we have:
\begin{align*}
[z^r]&\left(S_7(0\mid 1)+S_7(1\mid 0)-S_7(1\mid -1)+qS_7(2\mid 0)-qS_7(2\mid 1)\right)\\
&=\sum_{s\geq 0}\dfrac{q^{r^2-rs+s^2}}{(q)_r(q)_s(q)_{r+s+2}}\left(q^{r+2s+2}+(q^r+q^s-q^{r-s})(1-q^{r+s+2})\right)\\
&=\sum_{s\geq 0}\dfrac{q^{r^2-rs+s^2}}{(q)_r(q)_s(q)_{r+s+2}}\left(q^{s}-q^{r-s}(1-q^s)(1-q^{r+s+2})\right)\\
&=\sum_{s\geq 0}\dfrac{q^{r^2-rs+s^2+s}}{(q)_r(q)_s(q)_{r+s+2}}
-\sum_{s\geq 0}\dfrac{q^{r^2-rs+s^2+r-s}(1-q^s)(1-q^{r+s+2})}{(q)_r(q)_s(q)_{r+s+2}}
\end{align*}
Now observe that:
\begin{align*}
\sum_{s\geq 0}\dfrac{q^{r^2-rs+s^2+r-s}(1-q^s)(1-q^{r+s+2})}{(q)_r(q)_s(q)_{r+s+2}}
=\sum_{\mathbf{s\geq 1}}\dfrac{q^{r^2-rs+s^2+r-s}}{(q)_r(q)_{s-1}(q)_{r+s+1}}
=\sum_{\mathbf{s\geq 0}}\dfrac{q^{r^2-rs+s^2+s}}{(q)_r(q)_{s}(q)_{r+s+2}}.
\end{align*}

\end{proof}
We finally require the following --
\begin{lem}
For $A\in\ZZ$,
\begin{align}
-S_7(A\mid 1)+(1+zq^{A+1})S_7(A+2\mid 0)+zq^{A+2}S_7(A+3\mid 0)-qS_7(A+3\mid 1)=0.
\tag{\mbox{$R_3$}}
\label{eqn:m7r3}
\end{align}
\end{lem}
\begin{proof}
Let $A=0$.
Coefficient of $z^0$ in the LHS comes from $-S_7(0\mid 1)+S_7(2\mid0)-qS_7(3\mid 1)$
and equals:
\begin{align*}
\sum_{s\geq 0}&\dfrac{q^{s^2}(-q^s+1-q^{s+1})}{(q)_s(q)_{s+1}}
=
-\sum_{s\geq 0}\dfrac{q^{s^2+s}}{(q)_s(q)_{s+1}}
+\sum_{s\geq 0}\dfrac{q^{s^2}}{(q)_s^2}
=\dfrac{1}{(q)_\infty}-\dfrac{1}{(q)_\infty}=0.
\end{align*}
For $r\geq 0$, coefficient of $z^{r+1}$ is easily seen to equal:
\begin{align*}
\sum_{s\geq 0}&\dfrac{q^{r^2-rs+s^2+2r-s+1}(-q^s+q^{2r+2}-q^{3r+s+4})}{(q)_{r+1}(q)_s(q)_{r+s+2}}
+\sum_{s\geq 0}\dfrac{q^{r^2-rs+s^2}(q^{2r+1}+q^{3r+2})}{(q)_{r}(q)_s(q)_{r+s+1}}\\
&=\sum_{s\geq 0}\dfrac{q^{r^2-rs+s^2+3r+3}(q^{r-s}-q^{2r+2}-q^s-q^r+q^{2r+s+2})}{(q)_{r+1}(q)_s(q)_{r+s+2}}\\
&=-\sum_{s\geq 0}\dfrac{q^{r^2-rs+s^2+3r+3}(q^s-q^{r-s}(1-q^s)(1-q^{r+s+2}))}{(q)_{r+1}(q)_s(q)_{r+s+2}}\\
&=0,
\end{align*}
where the proof of the last step is as in the previous lemma.
\end{proof}

Now, the recurrence for $H_{(2,1,1)}$ amounts to proving
\begin{align*}
-S_7(0\mid 0)+S_7(0\mid 1)+S_7(1\mid  0)+q(1+z)S_7(1\mid 1)-S_7(2\mid  0)-qS_7(2\mid 1)+q(1-zq)S_7(3\mid 1)=0
\label{eqn:rec211}
\end{align*}
which can be deduced by considering:
\begin{align*}
-\ref{eqn:m7r1}(0)+zq\ref{eqn:m7r2}(1) - \ref{eqn:m7r3}(0).
\end{align*}

Recurrence for $H_{(2,2,0)}$ requires proving that:
\begin{align*}
-S_7(-1\mid  1)+(1+z)S_7(1\mid  0)+zqS_7(2\mid  0)-qS_7(2\mid 1)=0
\end{align*}
which is simply $\ref{eqn:m7r3}(-1)$.

Initial condition for $H_{(2,2,0)}$ can be quickly deduced from Corollary \ref{cor:initcond}.
Combined with Theorem \ref{thm:initcond}, we have established Conjecture \ref{conj:main}
for compositions of $4$.

\section{Moduli 8, 9, 10}
\label{sec:m8910}
In these moduli, we will observe some new phenomena which are not present in lower moduli,
especially, we will see ``non-terminal'' and ``terminal'' relations.

We shall completely prove the moduli $8$ and $10$ cases. For mod $9$, at the moment 
we do not have a proof, and thus at the end of this section, we present the conjectures.

\subsection{Mod 8}
We have following arrangement.
\begin{align}
\begin{array}{ll}
5,0,0 		& 		 \\
4,1,0/4,0,1 & 3,1,1  \\
3,2,0/3,0,2 & 2,2,1. \\
\hline
\end{array}
\end{align}
We have the following identities:
\begin{align}
H_{(5,0,0)}&=S_8(1,1\mid 1,1)\\
H_{(4,1,0)}&=S_8(0,1\mid 1,1)\\
H_{(3,2,0)}&=S_8(0,0\mid 1,1)\\
H_{(3,1,1)}&=S_8(0,1\mid 0,1)-qS_8(1,1\mid 1,1)\\
H_{(2,2,1)}&=S_8(0,0\mid 0,1)-qS_8(0,1\mid 1,1)\\
H_{(4,0,1)}&=S_8(1,1\mid 0,1)-q(1-z)S_8(2,1\mid 1,1)\\
H_{(3,0,2)}&=S_8(1,1\mid 0,0)-q(1-z)S_8(2,1\mid 0,1).
\end{align}

We now show that these give unique solutions to \eqref{eqn:Hrec} with $r=2$ and $\ell=5$.
The recursions for $H_{(5,0,0)}$, $H_{(4,0,1)}$ are immediately satisfied.
We have two relations emanating from the ``non-terminal'' variables
$r_1$ and $s_1$.
\begin{lem}
For $A,B,C,D\in\ZZ$, we have:
\begin{align}
S_8(A,B\mid C,D)-S_8(A+1,B-1\mid C,D)-zq^{A+1}S_8(A+2,B\mid C-1,D)&=0,
\tag{\mbox{$R_1$}}
\label{eqn:m8r1}\\
S_8(A,B\mid C,D)-S_8(A,B\mid C+1,D-1)-q^{C+1}S_8(A-1,B\mid C+2,D)&=0
\tag{\mbox{$R_2$}}
\label{eqn:m8r2}.
\end{align}
\end{lem}
\begin{proof}
For the first, observe that:
\begin{align*}
S_8&(A,B\mid C,D)-S_8(A+1,B-1\mid C,D)\\
&=\sum_{\substack{r_1\geq r_2\geq 0 \\ s_1\geq s_2\geq 0}}z^{r_1}
\dfrac{q^{r_1^2-r_1s_1+s_1^2+r_2^2+r_2s_2+s_2^2+Ar_1+Br_2+Cs_1+Ds_2}
(1-q^{r_1-r_2})}
{(q)_{r_1-r_2}(q)_{s_1-s_2}(q)_{r_2}(q)_{s_2}(q)_{r_2+s_2+1}}\\
&=\sum_{\substack{\mathbf{r_1>} r_2\geq 0 \\ s_1\geq s_2\geq 0}}z^{r_1}
\dfrac{q^{r_1^2-r_1s_1+s_1^2+r_2^2+r_2s_2+s_2^2+Ar_1+Br_2+Cs_1+Ds_2}}
{(q)_{\mathbf{r_1-r_2-1}}(q)_{s_1-s_2}(q)_{r_2}(q)_{s_2}(q)_{r_2+s_2+1}}\\
&=\sum_{\substack{\mathbf{r_1\geq} r_2\geq 0 \\ s_1\geq s_2\geq 0}}z^{r_1+1}
\dfrac{q^{r_1^2-r_1s_1+s_1^2+r_2^2+r_2s_2+s_2^2+\mathbf{(A+2)r_1}+Br_2+\mathbf{(C-1)s_1}+Ds_2+\mathbf{A+1}}}
{(q)_{r_1-r_2}(q)_{s_1-s_2}(q)_{r_2}(q)_{s_2}(q)_{r_2+s_2+1}}\\
&=zq^{A+1}S_8(A+2,B\mid C-1,D).
\end{align*}
The second relation is similar, and depends on the factor $(1-q^{s_1-s_2})$.
\end{proof}

Now we deduce two terminal relations. These depend on the 
quadratic $q^{r_2^2\boldsymbol{+}r_2s_2+s_2^2}$.
\begin{lem}
For $A,B,C,D\in\ZZ$ we have:
\begin{align}
S_8(A,B\mid C,0)-S_8(A,B\mid C,1)-qS_8(A,B+1\mid C,1)+qS_8(A,B+1\mid C+1,1)&=0
\tag{\mbox{$R_3$}}
\label{eqn:m8r3}\\
S_8(A, 0\mid B, C)-S_8(A, 1\mid B, C)-qS_8(A, 1\mid B, C+1)+qS_8(A+1, 1\mid B, C+1) &= 0
\tag{\mbox{$R_4$}}
\label{eqn:m8r4}
\end{align}
\end{lem}
\begin{proof}
For the first, we have:
\begin{align*}
S_8&(A,B\mid C,0)-S_8(A,B\mid C,1)-qS_8(A,B+1\mid C,1)+qS_8(A,B+1\mid C+1,1)\\
&=\sum_{\substack{r_1\geq r_2\geq0\\s_1\geq s_2\geq 0}}
z^{r_1}\dfrac{q^{r_1^2-r_1s_1+s_1^2+r_2^2+r_2s_2+s_2^2+Ar_1+Br_2+Cs_1}
}{(q)_{r_1-r_2}(q)_{s_1-s_2}(q)_{r_2}(q)_{s_2}(q)_{r_2+s_2+1}}
\left( 1-q^{s_2}-q^{r_2+s_2+1}+q^{r_2+s_1+1}\right)\\
&=\sum_{\substack{r_1\geq r_2\geq0\\s_1\geq s_2\geq 0}}
z^{r_1}\dfrac{q^{r_1^2-r_1s_1+s_1^2+r_2^2+r_2s_2+s_2^2+Ar_1+Br_2+Cs_1}
}{(q)_{r_1-r_2}(q)_{s_1-s_2}(q)_{r_2}(q)_{s_2}(q)_{r_2+s_2+1}}
\left((1-q^{s_2}) (1-q^{r_2+s_2+1}) - q^{r_2+2s_2+1}(1-q^{s_1-s_2})\right)\\
&=
\sum_{\substack{r_1\geq r_2\geq0\\ s_1\geq \mathbf{s_2\geq 1}}}
z^{r_1}\dfrac{q^{r_1^2-r_1s_1+s_1^2+r_2^2+r_2s_2+s_2^2+Ar_1+Br_2+Cs_1}
}{(q)_{r_1-r_2}(q)_{s_1-s_2}(q)_{r_2}(q)_{\mathbf{s_2-1}}(q)_{\mathbf{r_2+s_2}}}\\
&\quad\quad\quad-
\sum_{\substack{r_1\geq r_2\geq0\\ \mathbf{s_1> s_2}\geq 0}}
z^{r_1}\dfrac{q^{r_1^2-r_1s_1+s_1^2+r_2^2+r_2s_2+s_2^2+Ar_1+\mathbf{(B+1)r_2}+Cs_1+\mathbf{2s_2+1}}
}{(q)_{r_1-r_2}(q)_{\mathbf{s_1-s_2-1}}(q)_{r_2}(q)_{s_2}(q)_{r_2+s_2+1}}\\
&=
\sum_{\substack{r_1\geq r_2\geq0\\ s_1\geq \mathbf{s_2\geq 0}}}
z^{r_1}\dfrac{q^{r_1^2-r_1s_1+s_1^2+r_2^2+r_2s_2+s_2^2+\mathbf{(A-1)r_1}+\mathbf{(B+1)r_2}+\mathbf{(C+2)s_1}+\mathbf{2s_2}+\mathbf{C+2}}
}{(q)_{r_1-r_2}(q)_{s_1-s_2}(q)_{r_2}(q)_{\mathbf{s_2}}(q)_{\mathbf{r_2+s_2+1}}}\\
&\quad\quad\quad-
\sum_{\substack{r_1\geq r_2\geq0\\ \mathbf{s_1\geq s_2}\geq 0}}
z^{r_1}\dfrac{q^{r_1^2-r_1s_1+s_1^2+r_2^2+r_2s_2+s_2^2+\mathbf{(A-1)r_1}
+(B+1)r_2+\mathbf{(C+2)s_1}+2s_2+\mathbf{(C+2)}}
}{(q)_{r_1-r_2}(q)_{\mathbf{s_1-s_2}}(q)_{r_2}(q)_{s_2}(q)_{r_2+s_2+1}}\\
&=0.
\end{align*}

The second equation is proved similarly, and depends on the following
polynomial identity:
\begin{align*}
1-q^{r_2}-q^{r_2+s_2+1}+q^{r_1+r_2+s_2+1}
=(1-q^{r_2})(1-q^{r_2+s_2+1})-q^{2r_2+s_2+1}(1-q^{r_1-r_2}).
\end{align*}
\end{proof}

Now we prove the remaining recurrences.

Recurrence for $H_{(4,1,0)}$ amounts to proving:
\begin{align}
S_8(1, 0, 1, 1)-S_8(0, 1, 1, 1)+zqS_8(2, 1, 0, 1)=0,
\end{align}
which is nothing but $\ref{eqn:m8r1}(1,0,1,1)$.

Recurrence for $H_{(3,0,2)}$ amounts to proving:
\begin{align}
S_8(1, 0, 0, 1)-qS_8(1, 1, 1, 1)-S_8(1, 1, 0, 0)+qS_8(2, 1, 0, 1)=0
\end{align}
which is $\ref{eqn:m8r4}(1,0,0)-\ref{eqn:m8r3}(1,0,0)$.

Recurrence for $H_{(3,2,0)}$ is:
\begin{align}
-&S_8(0, 0, 1, 1)+S_8(1, 1, 0, 1)+(qz-1)S_8(2, 0, 0, 1)+S_8(2, 1, 0, 0)\notag\\
&\quad\quad-zq^2S_8(2, 1, 1, 1)+q(qz-1)S_8(3, 1, 0, 1)=0
\end{align}
which is:
\begin{align}
&- (1-zq)(\ref{eqn:m8r4}(2,0,0)-\ref{eqn:m8r3}(2,0,0))-\ref{eqn:m8r4}(1,1,0)
\notag\\
&\quad-\ref{eqn:m8r1}(0,1,1,0)-\ref{eqn:m8r4}(0,1,0)+\ref{eqn:m8r3}(0,0,1)+\ref{eqn:m8r2}(1,1,0,1)
\end{align}

Recurrence for $H_{(3,1,1)}$ is:
\begin{align}
&-S_8(0, 1, 0, 1)+S_8(1, 0, 0, 1)+S_8(1, 1, 1, 1)+(qz-1)S_8(2, 0, 0, 1)+(qz-1)S_8(2, 0, 1, 1)\notag\\
&\quad\quad+S_8(2, 1, 0, 0)-q(qz-1)S_8(2, 1, 1, 1)+q(qz-1)(qz+1)S_8(3, 1, 0, 1)\notag\\
&=0
\end{align}
which is:
\begin{align}
-\ref{eqn:m8r1}(0,1,0,1)+(1-zq)(\ref{eqn:m8r1}(1,1,1,1)-\ref{eqn:m8r4}(2,0,0)+\ref{eqn:m8r3}(2,0,0))-zq\ref{eqn:m8r2}(2,1,-1,1).
\end{align}

Recurrence for $H_{(2,2,1)}$ is:
\begin{align}
&-S_8(0, 0, 0, 1)+qS_8(0, 1, 1, 1)+S_8(1, 0, 0, 1)+S_8(1, 0, 1, 1)+S_8(1, 1, 0, 1)-qS_8(1, 1, 1, 1)\notag\\
&\quad\quad+(qz-1)S_8(2, 0, 0, 1)+(qz-1)S_8(2, 1, 0, 1)-q^2S_8(2, 1, 1, 1)z+(qz-1)(q^2z-1)S_8(3, 0, 0, 1)
\notag\\
&\quad\quad+(qz-1)S_8(3, 1, 0, 0)-zq^3(qz-1)S_8(3, 1, 1, 1)+q(qz-1)(q^2z-1)S_8(4, 1, 0, 1)
\notag\\
&=0
\end{align}
which can be obtained as:
\begin{align}
&q\ref{eqn:m8r3}(-1, 0, 2)
-q\ref{eqn:m8r3}(-1, 1, 2)
-q\ref{eqn:m8r4}(-1, 2, 0)
+q\ref{eqn:m8r4}(-1, 2, 1)
-zq\ref{eqn:m8r3}(2, 0, 0)
+zq^2\ref{eqn:m8r3}(3, 1, 0)
\notag\\
&\,\,
+zq\ref{eqn:m8r4}(2, 0, 0)
-q\ref{eqn:m8r1}(-1, 1, 2, 1)
-q^2\ref{eqn:m8r1}(-1, 2, 2, 1)
+q\ref{eqn:m8r1}(0, 1, 2, 1)
-q\ref{eqn:m8r1}(1, 2, 1, 1)
\notag\\
&\,\,
+q\ref{eqn:m8r1}(1, 2, 2, 1)
-q\ref{eqn:m8r2}(0, 1, 0, 2)
+q\ref{eqn:m8r2}(1, 1, 0, 2)
+zq\ref{eqn:m8r2}(2, 2, 0, 1)
-\ref{eqn:m8r1}(0, 1, 1, 0)
\notag\\
&\,\,
-\ref{eqn:m8r1}(0, 1, 1, 1)
+\ref{eqn:m8r1}(0, 1, 2, 0)
+\ref{eqn:m8r1}(0, 2, 1, 1)
-\ref{eqn:m8r1}(0, 2, 2, 0)
-\ref{eqn:m8r2}(0, 1, 0, 1)
\notag\\
&\,\,
+\ref{eqn:m8r2}(0, 1, 1, 1)
-\ref{eqn:m8r2}(0, 2, 1, 1)
+\ref{eqn:m8r2}(1, 0, 0, 1)
+\ref{eqn:m8r2}(1, 1, 0, 1)
+(q^2z-q)\ref{eqn:m8r4}(1, 2, 1)
\notag\\
&\,\,
-zq\ref{eqn:m8r2}(2, 1, 0, 1)
+(qz-2)\ref{eqn:m8r2}(2, 0, 0, 1)
+(qz-1)\ref{eqn:m8r3}(1, 1, 1)
+(-q^3z^2+q^2z+qz-1)\ref{eqn:m8r3}(3, 0, 0)
\notag\\
&\,\,
+(-qz+1)\ref{eqn:m8r4}(2, 0, 1)
+(q^3z^2-q^2z-qz+1)\ref{eqn:m8r4}(3, 0, 0)
+(-qz+2)\ref{eqn:m8r1}(1, 1, 1, 0)
\notag\\
&\,\,
+(q^2z-q)\ref{eqn:m8r2}(3, 1, 0, 2)
+(-q^2z+q)\ref{eqn:m8r2}(2, 1, 0, 2)
-\ref{eqn:m8r4}(0, 0, 1)
+\ref{eqn:m8r4}(1, 2, 0)
\end{align}
While not impossible, it is impractical to check by hand that this recurrence can be indeed obtained using the linear combination
given above. In Appendix \ref{app:verify}, we explain how to verify such equalities using a computer.

Using Theorem \ref{thm:initcond} for the initial conditions, we have established Conjecture \ref{conj:main}
for compositions of $5$.

\subsection{Mod 10}

The arrangement is as follows:
\begin{align}
\begin{array}{lll}
7,0,0 		& 		       &\\
6,1,0/6,0,1 & 5,1,1        & \\
5,2,0/5,0,2 & 4,2,1/4,1,2  & 3,2,2\\
\hline
4,3,0/4,0,3 & 3,3,1 &
\end{array}
\end{align}

Now, all the identities above the line are given by our truncation procedure:
\begin{align}
H_{(7,0,0)}&=S_{10}( 1,1\mid 1,1)\\
H_{(6,1,0)}&=S_{10}( 0,1\mid 1,1)\\
H_{(5,2,0)}&=S_{10}( 0,0\mid 1,1)\\
H_{(5,1,1)}&=S_{10}( 0,1\mid 0,1)-qS_{10}(1,1\mid 1,1)\\
H_{(4,2,1)}&=S_{10}( 0,0\mid 0,1)-qS_{10}(0,1\mid 1,1)\\
H_{(3,2,2)}&=S_{10}( 0,0\mid 0,0)-qS_{10}(0,1\mid 0,1)\\
H_{(6,0,1)}&=S_{10}( 1,1\mid 0,1)-q(1-z)S_{10}(2,1\mid 1,1)\\
H_{(5,0,2)}&=S_{10}( 1,1\mid 0,0)-q(1-z)S_{10}(2,1\mid 0,1)\\
H_{(4,1,2)}&=S_{10}( 0,1\mid 0,0)-qS_{10}(1,1\mid 0,1)
\end{align}

In order to find the missing identities 
for $H_{(4,3,0)}$, $H_{(4,0,3)}$ and $H_{(3,3,1)}$ we use 
Corteel--Welsh recursions to express them in terms of
the identities above the line.

Solving the recurrence for $H_{(5,2,0)}$ gets us:
\begin{align}
H_{(4,3,0)} = S_{10}(-1, 0\mid 1, 1)
            -S_{10}(0, 1\mid 0, 1)
       +(1-z)S_{10}(1, 0\mid 0, 1)
           +zqS_{10}(1, 1\mid 1, 1).
\end{align}
Solving the recurrence for $H_{(4,2,1)}$ gives:
\begin{align}
H_{(3,3,1)} &=    S_{10}(-1, 0\mid 0, 1)
                    -S_{10}(0, 1\mid 0, 0)
                    +(1-z)S_{10}(1, 0\mid 0, 0)               
                    +zqS_{10}(1, 1\mid 0, 1)\notag\\
           &\quad -qS_{10}(-1, 1\mid 1, 1)
                     -zS_{10}(0, 0\mid 1, 1),
\end{align}
Finally, to get $H_{(4,0,3)}$ one may solve the recurrence for
$H_{(4,0,3)}$ itself (which is what we do), 
or for $H_{(4,3,0)}$, or as indicated in Section \ref{sec:conj}
one may solve the recurrecence for $H_{(4,1,2)}$.
Solving the recurrence for $H_{(4,0,3)}$ leads to the 
shortest expression for $H_{(4,0,3)}$, so we use this, deviating
from the scheme sketched in Section \ref{sec:conj}.
We obtain --
\begin{align}
H_{(4,0,3)} &=   S_{10}(0, 0\mid 0, 1)
                 -qS_{10}(0, 1\mid 1, 1)
                    -zqS_{10}(1, 0\mid 1, 1)
                -S_{10}(1, 1\mid 0, 0) \notag \\
          &\quad  +(1-qz)S_{10}(2, 0\mid 0, 0)
                     +zq^2S_{10}(2, 1\mid 0, 1)
                     +zq S_{10}(2, 1\mid 0, 0)
\end{align}

With this, several recurrences are satisfied immediately, namely, the ones for $H_{(7,0,0)}$, $H_{(6,0,1)}$, $H_{(5,0,2)}$.
Also, the recurrences for $H_{(5,2,0)}$, $H_{(4,2,1)}$ and $H_{(4,0,3)}$ 
are satisfied by construction.
We still need to prove that the recurrences for $H_{(6,1,0)}$, $H_{(5,1,1)}$,
$H_{(4,1,2)}$, $H_{(3,2,2)}$,  $H_{(4,3,0)}$ and $H_{(3,3,1)}$.

Now we provide the relations that govern the sum-sides.
Exactly in analogy with the Mod-8 case, we have the terminal 
and non-terminal relations.
Relations corresponding to the non-terminal variables
$r_1$ and $s_1$ are given as follows.
\begin{lem}
For $A,B,C,D\in\ZZ$, we have:
\begin{align}
S_{10}(A,B\mid C,D)-S_{10}(A+1,B-1\mid C,D)-zq^{A+1}S_{10}(A+2,B\mid C-1,D)&=0,
\tag{\mbox{$R_1$}}
\label{eqn:m10r1}\\
S_{10}(A,B\mid C,D)-S_{10}(A,B\mid C+1,D-1)-q^{C+1}S_{10}(A-1,B\mid C+2,D)&=0
\tag{\mbox{$R_2$}}
\label{eqn:m10r2}.
\end{align}
\end{lem}

The terminal relations are:
\begin{lem}
For $A,B,C,D\in\ZZ$, we have:
\begin{align}
&S_{10}(A,B\mid C,D)-S_{10}(A,B\mid C,D+1)-qS_{10}(A,B+1\mid C,D+1)\notag\\
&\qquad\qquad+qS_{10}(A,B+1\mid C,D+2)-q^{C+D+2}S_{10}(A-1,B-1\mid C+2,D+2)=0,
\tag{\mbox{$R_3$}}
\label{eqn:m10r3}\\
&S_{10}(A,B\mid C,D)-S_{10}(A,B+1\mid C,D)-qS_{10}(A,B+1\mid C,D+1)\notag\\
&\qquad\qquad+qS_{10}(A,B+2\mid C,D+1)-zq^{A+B+2}S_{10}(A+2,B+2\mid C-1,D-1)=0
\tag{\mbox{$R_4$}}
\label{eqn:m10r4}.
\end{align}
\end{lem}

Now, the recurrence for $H_{(6,1,0)}$ 
amounts to:
\begin{align*}
zqS_{10}(2,1\mid 0,1) +S_{10}(1,0\mid 1,1)  -S_{10}(0,1\mid 1,1)&=0
\end{align*}
which is nothing but
$-\ref{eqn:m10r1}(0,1,1,1)$.

Recurrence for $H_{(5,1,1)}$ is equivalent to showing:
\begin{align*}
&q^3z^2S_{10}(3, 1\mid 0, 1) - q^2zS_{10}(3, 1\mid 0, 1) + qzS_{10}(2, 1\mid 0, 0) \\
&+ qzS_{10}(2, 0\mid 1, 1) - S_{10}(2, 0\mid 1, 1) + S_{10}(1, 1\mid 1, 1) + S_{10}(1, 0\mid 0, 1) - S_{10}(0, 1\mid 0, 1)
=0
\end{align*}
which can be obtained as:
\begin{align}
q\ref{eqn:m10r1}(-1,1,2,1)+\ref{eqn:m10r1}(0,1,1,0)-(1-qz)\ref{eqn:m10r1}(1,1,1,1)+\ref{eqn:m10r2}(0,1,0,1)-\ref{eqn:m10r2}(1,0,0,1)
\end{align}

Our feeling is that all relations sufficiently above the horizontal line
ought to be provable using only non-terminal relations, and the calculations above provide some evidence.

Recurrence for $H_{(4,1,2)}$ is equivalent to:
\begin{align*}
&q^4z^2S_{10}(3, 1\mid 0, 1) + q^3z^2S_{10}(3, 1\mid 0, 0) - q^3z^2S_{10}(3, 0\mid 0, 0) - q^3z^2S_{10}(2, 0\mid 1, 1) \\
&- q^2zS_{10}(2, 1\mid 1, 1) - q^2zS_{10}(1, 1\mid 1, 1) + qzS_{10}(3, 0\mid 0, 0) - qzS_{10}(2, 1\mid 0, 0)\\
& + qzS_{10}(2, 0\mid 0, 1) + qzS_{10}(1, 0\mid 0, 1) - S_{10}(2, 0\mid 0, 1) + S_{10}(1, 1\mid 0, 1)\\
& + S_{10}(1, 0\mid 0, 0) - S_{10}(0, 1\mid 0, 0)=0,
\end{align*}
 which is:
\begin{align}
&-zq\ref{eqn:m10r1}(0, 2, 0, 1)
+zq^2\ref{eqn:m10r1}(0, 2, 1, 1)
+zq^3\ref{eqn:m10r1}(-1, 2, 2, 2)
-zq^2\ref{eqn:m10r1}(0, 2, 0, 2)
\notag\\
&
-zq^2\ref{eqn:m10r1}(2, 2, 0, 0)
+zq^2\ref{eqn:m10r1}(-1, 2, 2, 1)
-zq\ref{eqn:m10r1}(2, 1, 0, 0)
+zq\ref{eqn:m10r1}(0, 2, 1, 0)
\notag\\
&
+q^4z^2\ref{eqn:m10r1}(3, 2, -1, 0)
+zq^3\ref{eqn:m10r1}(0, 1, 2, 2)
+zq^2\ref{eqn:m10r1}(0, 1, 2, 1)
\notag\\
&
-zq^2\ref{eqn:m10r1}(1, 2, 0, 2)
-\ref{eqn:m10r1}(0, 1, 0, 0)
+\ref{eqn:m10r1}(1, 1, 0, 1)
\notag\\
&
+(-q^2z^2-q^2z)\ref{eqn:m10r2}(2, 2, -1, 1)
+(-q^3z^2+q^2z)\ref{eqn:m10r2}(2, 2, -1, 2)
+(-q^3z^2+q^2z)\ref{eqn:m10r2}(3, 1, -1, 1)
\notag\\
&
-zq^2\ref{eqn:m10r2}(1, 0, 1, 2)
-zq\ref{eqn:m10r2}(2, 1, -1, 1)
-q^4z^2\ref{eqn:m10r2}(3, 1, -1, 2)
\notag\\
&
+zq^2\ref{eqn:m10r2}(0, 2, 0, 2)
+zq\ref{eqn:m10r2}(0, 2, 0, 1)
+q^3z^2\ref{eqn:m10r2}(3, 0, -1, 1)
\notag\\
&
-zq\ref{eqn:m10r3}(2, 1, 0, 0)
+zq\ref{eqn:m10r3}(1, 1, 1, 0)
-zq\ref{eqn:m10r3}(2, 1, -1, 0)
\notag\\
&
+zq\ref{eqn:m10r4}(1, 0, 0, 1)
-q^3z^2\ref{eqn:m10r4}(3, 0, -1, 1)
+zq\ref{eqn:m10r4}(2, 0, 0, 1)
\label{eqn:rec412lincomb}
\end{align}

The proof that recurrences for $H_{(3,2,2)}$, $H_{(4,3,0)}$ and $H_{(3,3,1)}$
are consequences of relations \eqref{eqn:m10r1}--\eqref{eqn:m10r4} is much 
longer than \eqref{eqn:rec412lincomb}. 
Each of these uses a linear combination of  hundreds of terms involving \eqref{eqn:m10r1}--\eqref{eqn:m10r4}.
We thus defer the check to a computer, as explained in Appendix \ref{app:verify}.

Combined with Theorem \ref{thm:initcond} and Corollary \ref{cor:initcond} we have established Conjecture \ref{conj:main}
for compositions of $7$.

Setting $z\mapsto 1$ and using \eqref{eqn:Hcprod}, we arrive at the following sum-product identities,
some of which are new. 

\begin{cor}
We have the following identities, with each $\sum$ denoting a sum over $r_1,r_2,s_1,s_2\geq 0$:
\begin{align}
\sum
&\dfrac{q^{r_1^2-r_1s_1+s_1^2+r_2^2-r_2s_2+s_2^2\,+\,r_1+r_2+s_1+s_2}}{(q)_{r_1-r_2}(q)_{s_1-s_2}(q)_{r_2}(q)_{s_2}(q)_{r_2+s_2+1}}
=\dfrac{1}{(q)_\infty}\dfrac{1}{\theta(q^2,q^3,q^3,q^4,q^4,q^5;\,q^{10})}
=\dfrac{\chi(\Omega(7\Lambda_0))}{(q)_\infty}
\label{eqn:700}
\\
\sum
&\dfrac{q^{r_1^2-r_1s_1+s_1^2+r_2^2-r_2s_2+s_2^2\,+\,r_2+s_1+s_2}}{(q)_{r_1-r_2}(q)_{s_1-s_2}(q)_{r_2}(q)_{s_2}(q)_{r_2+s_2+1}}
=\dfrac{1}{(q)_\infty}\dfrac{1}{\theta(q,q^2,q^3,q^4,q^4,q^5;\,q^{10})}
=\dfrac{\chi(\Omega(6\Lambda_0+\Lambda_1))}{(q)_\infty}
\label{eqn:610}
\\
\sum
&\dfrac{q^{r_1^2-r_1s_1+s_1^2+r_2^2-r_2s_2+s_2^2\,+\,s_1+s_2}}{(q)_{r_1-r_2}(q)_{s_1-s_2}(q)_{r_2}(q)_{s_2}(q)_{r_2+s_2+1}}
=\dfrac{1}{(q)_\infty}\dfrac{1}{\theta(q,q^2,q^2,q^3,q^4,q^5;\,q^{10})}
=\dfrac{\chi(\Omega(5\Lambda_0+2\Lambda_1))}{(q)_\infty}
\label{eqn:520}
\\
\sum
&\dfrac{q^{r_1^2-r_1s_1+s_1^2+r_2^2-r_2s_2+s_2^2\,+\,r_2+s_2}(1-q^{1+r_1+s_1})}{(q)_{r_1-r_2}(q)_{s_1-s_2}(q)_{r_2}(q)_{s_2}(q)_{r_2+s_2+1}}
=
\dfrac{1}{(q)_\infty}\dfrac{1}{\theta(q,q,q^3,q^3,q^4,q^5;\,q^{10})}
=\dfrac{\chi(\Omega(5\Lambda_0+\Lambda_1+\Lambda_2))}{(q)_\infty}
\label{eqn:511}
\\
\sum
&\dfrac{q^{r_1^2-r_1s_1+s_1^2+r_2^2-r_2s_2+s_2^2\,+\,s_2}(1-q^{1+r_2+s_1})}{(q)_{r_1-r_2}(q)_{s_1-s_2}(q)_{r_2}(q)_{s_2}(q)_{r_2+s_2+1}}
=
\dfrac{1}{(q)_\infty}\dfrac{1}{\theta(q,q,q^2,q^3,q^4,q^4;\,q^{10})}
=\dfrac{\chi(\Omega(4\Lambda_0+2\Lambda_1+\Lambda_2))}{(q)_\infty}
\label{eqn:421}
\\
\sum
&\dfrac{q^{r_1^2-r_1s_1+s_1^2+r_2^2-r_2s_2+s_2^2}(1-q^{1+r_2+s_2})}{(q)_{r_1-r_2}(q)_{s_1-s_2}(q)_{r_2}(q)_{s_2}(q)_{r_2+s_2+1}}
=
\dfrac{1}{(q)_\infty}\dfrac{1}{\theta(q,q,q^2,q^2,q^4,q^5;\,q^{10})}
=\dfrac{\chi(\Omega(3\Lambda_0+2\Lambda_1+2\Lambda_2))}{(q)_\infty}
\label{eqn:322}
\\
\sum
&\dfrac{q^{r_1^2-r_1s_1+s_1^2+r_2^2-r_2s_2+s_2^2}(q^{-r_1+s_1+s_2}-q^{r_2+s_2}+q^{1+r_1+r_2+s_1+s_2})}
{(q)_{r_1-r_2}(q)_{s_1-s_2}(q)_{r_2}(q)_{s_2}(q)_{r_2+s_2+1}}\notag\\
&=
\sum
\dfrac{q^{r_1^2-r_1s_1+s_1^2+r_2^2-r_2s_2+s_2^2}
}{(q)_{r_1-r_2}(q)_{s_1-s_2}(q)_{r_2}(q)_{s_2}(q)_{r_2+s_2+1}}\notag\\
&\qquad\times
(
q^{s_2}-q^{1+r_2+s_1+s_2}-q^{1+r_1+s_1+s_2}-q^{r_1+r_2}
+(1-q)q^{2r_1}+q^{2+2r_1+r_2+s_2}+q^{1+2r_1+r_2}
)\notag\\
&=
\dfrac{1}{(q)_\infty}\dfrac{1}{\theta(q,q^2,q^2,q^3,q^3,q^4;\,q^{10})}
=\dfrac{\chi(\Omega(4\Lambda_0+3\Lambda_1))}{(q)_\infty}
\label{eqn:430}
\\
\sum
&\dfrac{q^{r_1^2-r_1s_1+s_1^2+r_2^2-r_2s_2+s_2^2}
(
q^{-r_1+s_2}-q^{r_2}+q^{1+r_1+r_2+s_2}-q^{1-r_1+r_2+s_1+s_2}-q^{s_1+s_2}
)
}{(q)_{r_1-r_2}(q)_{s_1-s_2}(q)_{r_2}(q)_{s_2}(q)_{r_2+s_2+1}}
\notag\\
&=
\dfrac{1}{(q)_\infty}\dfrac{1}{\theta(q,q,q^2,q^3,q^3,q^5;\,q^{10})}
=\dfrac{\chi(\Omega(3\Lambda_0+3\Lambda_1+\Lambda_2))}{(q)_\infty}
\label{eqn:331}
\end{align}
\end{cor}
Of these, \eqref{eqn:700}, \eqref{eqn:610}, \eqref{eqn:520}, \eqref{eqn:511}, \eqref{eqn:322}  were proved in \cite{AndSchWar} using their
then newly invented $\typeA_2$ Bailey lemma.

\subsection{Mod 9}

The arrangement is as follows.
\begin{align}
\begin{array}{lll}
6,0,0 		& 		       &\\
5,1,0/5,0,1 & 4,1,1        & \\
4,2,0/4,0,2 & 3,2,1/3,1,2  & 2,2,2\\
\hline
3,3,0.
\end{array}
\end{align}
As always,  identities above the line are given by our truncation procedure:
\begin{align}
H_{(6,0,0)}&=S_{9}( 1,1\mid 1,1)\\
H_{(5,1,0)}&=S_{9}( 0,1\mid 1,1)\\
H_{(4,2,0)}&=S_{9}( 0,0\mid 1,1)\\
H_{(4,1,1)}&=S_{9}( 0,1\mid 0,1)-qS_{9}(1,1\mid 1,1)\\
H_{(3,2,1)}&=S_{9}( 0,0\mid 0,1)-qS_{9}(0,1\mid 1,1)\\
H_{(2,2,2)}&=S_{9}( 0,0\mid 0,0)-qS_{9}(0,1\mid 0,1)\\
H_{(5,0,1)}&=S_{9}( 1,1\mid 0,1)-q(1-z)S_{9}(2,1\mid 1,1)\\
H_{(4,0,2)}&=S_{9}( 1,1\mid 0,0)-q(1-z)S_{9}(2,1\mid 0,1)\\
H_{(3,1,2)}&=S_{9}( 0,1\mid 0,0)-qS_{9}(1,1\mid 0,1).
\end{align}

For $H_{(3,3,0)}$, we have the following guess:
\begin{align}
H_{(3,3,0)}=S_9(0,1\mid 0,1)-q(1+z)S_9(1,1\mid 1,1)
\end{align}

Inspired by the mod $8$ and mod $10$ cases above,
we conjecture that the following relations are enough to prove
all the required recurrences, however at the moment we do not
have a full proof of this.

Nonetheless, it is easy to see that the recurrences for $H_{(6,0,0)}$,
$H_{(5,0,1)}$, $H_{(4,0,2)}$ are satisfied immediately.

We have the non-terminal relations:
\begin{lem}
For $A,B,C,D\in\ZZ$, we have:
\begin{align}
S_9(A,B\mid C,D)-S_9(A+1,B-1\mid C,D)-zq^{A+1}S_9(A+2,B\mid C-1,D)&=0,
\tag{\mbox{$R_1$}}
\label{eqn:m9r1}\\
S_9(A,B\mid C,D)-S_9(A,B\mid C+1,D-1)-q^{C+1}S_9(A-1,B\mid C+2,D)&=0
\tag{\mbox{$R_2$}}
\label{eqn:m9r2}.
\end{align}
\end{lem}

We have the terminal relations:
\begin{lem}
\begin{align}
S_9(A,B&\mid C,D) -(1+q)S_9(A,B+1\mid C,D+1)  + q S_9(A,B+2\mid C,D+2) \notag\\
&- zq^{2+A+B}S_9(A+2,B+2\mid C-1,D-1) - q^{2+C+D}S_9(A-1,B+2\mid C+2,D+2)
=0
\tag{\mbox{$R_3$}}
\label{eqn:m9r3}\\
S_9(A,B&\mid C,D) -(1+q)S_9(A,B+1\mid C,D+1)  + q S_9(A,B+2\mid C,D+2) \notag\\
&- zq^{2+A+B}S_9(A+2,B+2\mid C-1,D+2) - q^{2+C+D}S_9(A-1,B-1\mid C+2,D+2)
=0.
\tag{\mbox{$R_4$}}
\label{eqn:m9r4}
\end{align}
\end{lem}
\begin{proof}
The first one can be proved as follows:
\begin{align*}
S_9&(A,B\mid C,D) -(1+q)S_9(A,B+1\mid C,D+1)  + q S_9(A,B+2\mid C,D+2)\notag\\
&=\sum_{r_1,r_2,s_1,s_2\geq 0}
z^{r_1}\dfrac{q^{r_1^2-r_1s_1+s_1^2+r_2^2-r_2s_2+s_2^2+Ar_1+Br_2+Cs_1+Ds_2}(1-q^{r_2+s_2})(1-q^{r_2+s_2+1})}
{(q)_{r_1-r_2}(q)_{s_1-s_2}(q)_{r_2+s_2}(q)_{r_2+s_2+1}}
\qbin{r_2+s_2}{r_2}_{q^3}\\
&=
\sum_{r_1,r_2,s_1,s_2\geq 0}
z^{r_1}\dfrac{q^{r_1^2-r_1s_1+s_1^2+r_2^2-r_2s_2+s_2^2+Ar_1+Br_2+Cs_1+Ds_2}(1-q^{r_2+s_2})(1-q^{r_2+s_2+1})}
{(q)_{r_1-r_2}(q)_{s_1-s_2}(q)_{r_2+s_2}(q)_{r_2+s_2+1}}\notag\\
&\quad\quad\times \left(q^{3r_2}\qbin{r_2+s_2-1}{r_2}_{q^3}+\qbin{r_2+s_2-1}{r_2-1}_{q^3}\right)
\\
&=
\sum_{r_1,r_2,s_1\geq 0, \mathbf{s_2>0}}
z^{r_1}\dfrac{q^{r_1^2-r_1s_1+s_1^2+r_2^2-r_2s_2+s_2^2+Ar_1+(B+3)r_2+Cs_1+Ds_2}(1-q^{r_2+s_2})(1-q^{r_2+s_2+1})}
{(q)_{r_1-r_2}(q)_{s_1-s_2}(q)_{r_2+s_2}(q)_{r_2+s_2+1}}
\qbin{r_2+s_2-1}{r_2}_{q^3}\notag\\
&\quad+\sum_{r_1,s_1,s_2\geq 0, \mathbf{r_2>0}}
z^{r_1}\dfrac{q^{r_1^2-r_1s_1+s_1^2+r_2^2-r_2s_2+s_2^2+Ar_1+Br_2+Cs_1+Ds_2}(1-q^{r_2+s_2})(1-q^{r_2+s_2+1})}
{(q)_{r_1-r_2}(q)_{s_1-s_2}(q)_{r_2+s_2}(q)_{r_2+s_2+1}}
\qbin{r_2+s_2-1}{r_2-1}_{q^3}\notag
\\
&=
\sum_{r_1,r_2,s_1\geq 0, s_2>0}
z^{r_1}\dfrac{q^{r_1^2-r_1s_1+s_1^2+r_2^2-r_2s_2+s_2^2+Ar_1+(B+3)r_2+Cs_1+Ds_2}}
{(q)_{r_1-r_2}(q)_{s_1-s_2}(q)_{r_2+s_2-1}(q)_{r_2+s_2}}
\qbin{r_2+s_2-1}{r_2}_{q^3}\notag\\
&\quad+\sum_{r_1,s_1,s_2\geq 0, r_2>0}
z^{r_1}\dfrac{q^{r_1^2-r_1s_1+s_1^2+r_2^2-r_2s_2+s_2^2+Ar_1+Br_2+Cs_1+Ds_2}}
{(q)_{r_1-r_2}(q)_{s_1-s_2}(q)_{r_2+s_2-1}(q)_{r_2+s_2}}
\qbin{r_2+s_2-1}{r_2-1}_{q^3}\notag
\\
&=
q^{2+C+D}\sum_{r_1,r_2,s_1,s_2\geq 0}
z^{r_1}\dfrac{q^{r_1^2-r_1s_1+s_1^2+r_2^2-r_2s_2+s_2^2+(A-1)r_1+(B+2)r_2+(C+2)s_1+(D+2)s_2}}
{(q)_{r_1-r_2}(q)_{s_1-s_2}(q)_{r_2+s_2}(q)_{r_2+s_2+1}}
\qbin{r_2+s_2}{r_2}_{q^3}\notag\\
&+
zq^{2+A+B}\sum_{r_1,r_2,s_1,s_2\geq 0}
z^{r_1}\dfrac{q^{r_1^2-r_1s_1+s_1^2+r_2^2-r_2s_2+s_2^2+(A+2)r_1+(B+2)r_2+(C-1)s_1+(D-1)s_2}}
{(q)_{r_1-r_2}(q)_{s_1-s_2}(q)_{r_2+s_2}(q)_{r_2+s_2+1}}
\qbin{r_2+s_2}{r_2}_{q^3}\notag
\end{align*}
as required.

\eqref{eqn:m9r4} can be deduced similarly,
but we use:
\begin{align*}
\qbin{r_2+s_2}{r_2}_{q^3}=q^{3s_2}\qbin{r_2+s_2-1}{r_2-1}_{q^3}+\qbin{r_2+s_2-1}{r_2}_{q^3}.
\end{align*}
\end{proof}

\section{On proofs for higher moduli}
\label{sec:higher}
We now speculate on the shape of proofs for higher moduli.

Fix $k\geq 3$, let $m\in 3k+\{-1,0,+1\}$ and let $\ell=m-3$. 
As before, let $c$ be a length $3$ composition of $\ell$.
Recall the notation $\delta_j$ \eqref{eqn:dvect}.
It is very easy to generalize the ``non-terminal'' relations
\eqref{eqn:m8r1} and \eqref{eqn:m8r2} in moduli $8$, $9$ and $10$ to higher moduli.
\begin{lem} We have the following relations for $1\leq i\leq k-2$:
\begin{align}
S_m&(\rho\mid\sigma) - S_m(\rho + \delta_i - \delta_{i+1} \mid \sigma ) 
\notag\\
&- zq^{i+\rho_1+\rho_2+\cdots + \rho_i}S_m(\rho + 2\delta_1+2\delta_2+\cdots +2 \delta_i\mid \sigma - \delta_1-\delta_2-\cdots-\delta_i)
=0
\tag{\mbox{$R_1^{(i)}$}}
\label{eqn:R1gen}\\
S_m&(\rho\mid\sigma) - S_m(\rho  \mid \sigma + \delta_i - \delta_{i+1}) \notag\\
&- q^{i+\sigma_1+\sigma_2+\cdots + \sigma_i}S_m(\rho - \delta_1-\delta_2-\cdots-\delta_i\mid \sigma  + 2\delta_1+2\delta_2+\cdots +2 \delta_i)
=0
\tag{\mbox{$R_2^{(i)}$}}
\label{eqn:R2gen}
\end{align}
\end{lem}

We also have the ``terminal'' relations, but these now depend on the congruence class of $m$ modulo $3$.
These are proved in analogy with the corresponding relations in moduli $8$, $9$, $10$.
\begin{lem}
\begin{enumerate}
\item If $m=3k+1$ we have:
\begin{align}
S_{3k+1}&(\rho\mid\sigma) 
- S_{3k+1}(\rho\mid\sigma+\delta_{k-1}) 
- qS_{3k+1}(\rho+\delta_{k-1}\mid\sigma+\delta_{k-1}) 
+ qS_{3k+1}(\rho+\delta_{k-1}\mid\sigma+2\delta_{k-1}) \notag\\
&- q^{k-1+\sigma_1+\sigma_2+\cdots+\sigma_{k-1}}
S_{3k+1}(\rho-\delta_1-\delta_2-\cdots-\delta_{k-1}\mid
\sigma+2\delta_1+2\delta_2+\cdots+2\delta_{k-1}) =0
\tag{\mbox{$R_3$}}
\label{eqn:R3mod10gen}
\\
S_{3k+1}&(\rho\mid\sigma) 
- S_{3k+1}(\rho+\delta_{k-1}\mid\sigma) 
- qS_{3k+1}(\rho+\delta_{k-1}\mid\sigma+\delta_{k-1}) 
+ qS_{3k+1}(\rho+2\delta_{k-1}\mid\sigma+\delta_{k-1}) \notag\\
&- zq^{k-1+\rho_1+\rho_2+\cdots+\rho_{k-1}}
S_{3k+1}(\rho+2\delta_1+2\delta_2+\cdots+2\delta_{k-1}\mid
\sigma-\delta_1-\delta_2-\cdots-\delta_{k-1}) =0
\tag{\mbox{$R_4$}}
\label{eqn:R4mod10gen}
\end{align}
\item If $m=3k-1$ and if $\sigma_{k-1}=0$, we have:
\begin{align}
S_{3k-1}&(\rho\mid\sigma)-S_{3k-1}(\rho\mid \sigma+\delta_{k-1})
-qS_{3k-1}(\rho+\delta_{k-1}\mid \sigma+\delta_{k-1})\notag\\
&+qS_{3k-1}(\rho+\delta_{k-1}\mid \sigma+\delta_{k-2}+\delta_{k-1})=0
\tag{\mbox{$R_3$}}
\label{eqn:R3mod8gen}
\end{align}
If $m=3k-1$ and if $\rho_{k-1}=0$, we have:
\begin{align}
S_{3k-1}&(\rho\mid\sigma)-S_{3k-1}(\rho+\delta_{k-1}\mid \sigma)
-qS_{3k-1}(\rho+\delta_{k-1}\mid \sigma+\delta_{k-1})\notag\\
&+qS_{3k-1}(\rho+\delta_{k-2}+\delta_{k-1}\mid \sigma+\delta_{k-1})=0
\tag{\mbox{$R_4$}}
\label{eqn:R4mod8gen}
\end{align}
\item Finally, if $m=3k$, we have:
\begin{align}
S_{3k}&(\rho\mid\sigma) - (1+q)S_{3k}(\rho+\delta_{k-1}\mid \sigma+\delta_{k-1})
+qS_{3k}(\rho+2\delta_{k-1}\mid \sigma+2\delta_{k-1})\notag\\
&-zq^{k-1+\rho_1+\rho_2+\cdots+\rho_{k-1}}S_{3k}(\rho+2\delta_1+2\delta_2+\cdots+2\delta_{k-1}\mid \sigma-\delta_1-\delta_2-\cdots-\delta_{k-1})\notag\\
&-q^{k-1+\sigma_1+\sigma_2+\cdots+\sigma_{k-1}}S_{3k}(\rho-\delta_1-\delta_2-\cdots-\delta_{k-2}+2\delta_{k-1}\mid \sigma+2\delta_1+2\delta_2+\cdots+2\delta_{k-1})=0
\tag{\mbox{$R_3$}}
\label{eqn:R3mod9gen}\\
S_{3k}&(\rho\mid\sigma) - (1+q)S_{3k}(\rho+\delta_{k-1}\mid \sigma+\delta_{k-1})
+qS_{3k}(\rho+2\delta_{k-1}\mid \sigma+2\delta_{k-1})\notag\\
&-zq^{k-1+\rho_1+\rho_2+\cdots+\rho_{k-1}}S_{3k}(\rho+2\delta_1+2\delta_2+\cdots+2\delta_{k-1}\mid \sigma-\delta_1-\delta_2-\cdots-\delta_{k-2}+2\delta_{k-1})\notag\\
&-q^{k-1+\sigma_1+\sigma_2+\cdots+\sigma_{k-1}}
S_{3k}(\rho-\delta_1-\delta_2-\cdots-\delta_{k-1}\mid \sigma+2\delta_1+2\delta_2+\cdots+2\delta_{k-1})
=0
\tag{\mbox{$R_4$}}
\label{eqn:R4mod9gen}
\end{align}
\end{enumerate}
\end{lem}

\begin{conj}
In each modulus, the relations $R_1$--$R_4$ above are enough to prove the recurrences.
\end{conj}
\appendix

\section{Verifying the recurrences}
\label{app:verify}

We use SAGE \cite{sagemath} to verify our recurrences.
For Mathematica implementations of useful tools related to cylindric partitions,
see Ablinger and Uncu's package \texttt{qFunctions} \cite{AblUnc}.

In the ancillary files section of this paper on \texttt{arXiv}, 
we include all relevant files required for proofs.

Below, we show the verification process for our mod 10 identities. 
First, we need some basic setup:
\begin{python}
from sage.combinat.q_analogues import q_pochhammer as po
\end{python}
\begin{python}
var('q');
var('z');
function('S');
\end{python}

It is beneficial to define and use canonical forms of compositions $(c_0,c_1,c_2)$.
We say that $(c_0,c_1,c_2)$  is in a canonical form if
$c_0$ is the biggest part and if this part is repeated, then $c_1$ is also 
equal to $c_0$. 
Using the rotational symmetry of $H_{(c_0,c_1,c_2)}(z,q)$ (see \eqref{eqn:Frotsym}),
we may assume that $(c_0,c_1,c_2)$ is in a canonical form.
As examples, the canonical form of $(4,0,9)$ is $(9,4,0)$, the canonical form of $(4,0,4)$ is $(4,4,0)$, etc.
The following function inputs a composition $c$ and outputs its canonical form.
\begin{python}
def canform(c):
    m=max(c)
    s=c.index(m)
    if c[0]==c[2] or c[1]==c[2] or c[0]==c[1]:
        return(tuple(sorted(c,reverse=True)))
    if s==1:
        return((c[1],c[2],c[0]))
    if s==2:
        return((c[2],c[0],c[1]))       
    return(c)
\end{python}

Now we encode the Corteel--Welsh recursion (see also \cite{AblUnc}).
The following function computes the composition $c(J)$ when given $c$ and $J\subseteq I_c$.
\begin{python}
def cJ(c,J):
    ans=[i for i in c]
    for i in [1,2]:
        if (i in J) and not(i-1 in J):
            ans[i]=c[i]-1
        if not(i in J) and ((i-1) in J):
            ans[i]=c[i]+1
    if (0 in J) and not(2 in J):
        ans[0]=c[0]-1
    if not(0 in J) and (2 in J):
        ans[0]=c[0]+1
    return tuple(ans) 
\end{python}
With this, we now easily compute the recurrence using --
\begin{python}
def recH(c):
    Ic=[i for i in range(3) if c[i]>0]
    J=[tuple(x) for x in Subsets(Ic) if len(x)>0]
    return sum([-(-1)^len(j)*po(len(j)-1,z*q)*(H[canform(cJ(c,j))](len(j))) for j in J])
\end{python}
 
We should now define our $H_c$ functions.
In the following, $j$ corresponds to the shift $z\mapsto zq^j$, see \eqref{eqn:Sshift}.
\begin{python}
H = {
    (7,0,0): lambda j: S( 1+j,1,1,1),
    (6,1,0): lambda j: S(   j,1,1,1),
    (5,2,0): lambda j: S(   j,0,1,1),

    (5,1,1): lambda j: S(   j,1,0,1)-q*S(1+j,1,1,1),
    (4,2,1): lambda j: S(   j,0,0,1)-q*S(  j,1,1,1),

    (3,2,2): lambda j: S(   j,0,0,0)-q*S(  j,1,0,1),

    (6,0,1): lambda j: S( 1+j,1,0,1)-q*(1-z*q^j)*S(2+j,1,1,1),
    (5,0,2): lambda j: S( 1+j,1,0,0)-q*(1-z*q^j)*S(2+j,1,0,1),

    (4,1,2): lambda j: S(  j,1,0,0)-q*S(1+j,1,0,1),

    (4,3,0): lambda j: S(j-1, 0, 1, 1)-S(j, 1, 0, 1)+(-z*q^j+1)*S(1+j, 0, 0, 1)
                        +q*S(1+j, 1, 1, 1)*z*q^j,

    (3,3,1): lambda j: S(j-1, 0, 0, 1)-q*S(j-1, 1, 1, 1)-S(j, 0, 1, 1)*z*q^j
                        -S(j, 1, 0, 0)+(-q^j*z+1)*S(1+j, 0, 0, 0)+q^(1+j)*S(1+j, 1, 0, 1)*z,

    (4,0,3): lambda j:  S(j, 0, 0, 1)-q*S(j, 1, 1, 1)-S(1+j, 0, 1, 1)*z*q*q^j
                        -S(1+j, 1, 0, 0)+(-q*z*q^j+1)*S(2+j, 0, 0, 0) 
                        +q^2*S(2+j, 1, 0, 1)*z*q^j+S(2+j, 1, 0, 0)*z*q*q^j
}
\end{python}

The next function computes what needs to be proved to equal $0$
in order to establish the recurrence for $c$.
\begin{python}
def checkrec(c):
    return expand(recH(c) - H[c](0))
\end{python}
We determine all compositions of $n$ in standard form:
\begin{python}
def allprofiles(n):
    c=tuple(Compositions(n+3,length=3))
    c=(canform(tuple(i)) for i in c)
    c=((i[0]-1,i[1]-1,i[2]-1) for i in c)
    return(tuple(Set(c)))
\end{python}
and run through each of them, checking if their recurrences are satisfied.
\begin{python}
for c in allprofiles(7):
    print(c)
    print(checkrec(c))
    print("=======")
\end{python}

After excuting all of this, we see the output:
\begin{python}
(5, 2, 0)
0
=======
(7, 0, 0)
0
=======
(4, 3, 0)
q^4*z^2*S(3, 1, 0, 1) - q^3*z^2*S(3, 0, 0, 0) - q^3*z^2*S(2, 0, 1, 1) 
+ q^2*z*S(3, 1, 0, 0) - q^2*z*S(1, 1, 1, 1) + q*z*S(3, 0, 0, 0) - q*z*S(2, 1, 0, 0) 
- q*z*S(1, 1, 1, 1) + q*z*S(1, 0, 0, 1) - q*S(1, 1, 1, 1) 
+ z*S(1, 0, 0, 1) + S(0, 1, 0, 1) - S(-1, 0, 1, 1)
=======
(6, 0, 1)
0
=======
(3, 3, 1)
q^7*z^3*S(4, 1, 0, 1) - q^6*z^3*S(4, 0, 0, 0) - q^6*z^3*S(3, 0, 1, 1) 
- q^6*z^2*S(4, 1, 0, 1) + q^5*z^2*S(4, 0, 0, 0) + q^5*z^2*S(3, 0, 1, 1) 
+ q^4*z^2*S(4, 1, 0, 0) - q^4*z^2*S(2, 1, 1, 1) + q^3*z^2*S(4, 0, 0, 0) 
- q^3*z^2*S(3, 1, 0, 0) + q^3*z^2*S(2, 0, 0, 1) - q^3*z*S(4, 1, 0, 0) 
+ q^3*z*S(2, 1, 1, 1) - q^2*z*S(4, 0, 0, 0) + q^2*z*S(3, 1, 0, 0) 
- q^2*z*S(2, 1, 0, 1) - q^2*z*S(2, 0, 0, 1) + q*z*S(2, 0, 0, 0) 
- q*z*S(1, 1, 0, 1) + q*S(2, 1, 1, 1) - q*S(1, 1, 0, 1) + z*S(1, 0, 0, 0) 
+ z*S(0, 0, 1, 1) + q*S(-1, 1, 1, 1) - S(2, 0, 0, 0) - S(1, 1, 0, 1) 
+ S(1, 1, 0, 0) + S(0, 1, 0, 0) + S(0, 0, 1, 1) - S(-1, 0, 0, 1)
=======
(4, 2, 1)
0
=======
(3, 2, 2)
-q^3*z^2*S(2, 0, 1, 1) + q^2*z*S(2, 0, 1, 1) - q^2*z*S(1, 1, 1, 1) - q*z*S(1, 0, 1, 1) 
+ q*z*S(1, 0, 0, 1) + q*S(2, 1, 0, 1) - q*S(1, 1, 0, 1) - q*S(0, 1, 1, 1) 
+ q*S(0, 1, 0, 1) - S(1, 1, 0, 0) + S(1, 0, 0, 0) + S(0, 0, 0, 1) - S(0, 0, 0, 0)
=======
(5, 0, 2)
0
=======
(6, 1, 0)
q*z*S(2, 1, 0, 1) + S(1, 0, 1, 1) - S(0, 1, 1, 1)
=======
(4, 1, 2)
q^4*z^2*S(3, 1, 0, 1) + q^3*z^2*S(3, 1, 0, 0) - q^3*z^2*S(3, 0, 0, 0) 
- q^3*z^2*S(2, 0, 1, 1) - q^2*z*S(2, 1, 1, 1) - q^2*z*S(1, 1, 1, 1) + q*z*S(3, 0, 0, 0) 
- q*z*S(2, 1, 0, 0) + q*z*S(2, 0, 0, 1) + q*z*S(1, 0, 0, 1) - S(2, 0, 0, 1) 
+ S(1, 1, 0, 1) + S(1, 0, 0, 0) - S(0, 1, 0, 0)
=======
(4, 0, 3)
0
=======
(5, 1, 1)
q^3*z^2*S(3, 1, 0, 1) - q^2*z*S(3, 1, 0, 1) + q*z*S(2, 1, 0, 0) + q*z*S(2, 0, 1, 1) 
- S(2, 0, 1, 1) + S(1, 1, 1, 1) + S(1, 0, 0, 1) - S(0, 1, 0, 1)
=======
\end{python}
which means that the recurrences for $H_{(5,2,0)}$, $H_{(7,0,0)}$, $H_{(6,0,1)}$, $H_{(4,2,1)}$, $H_{(5,0,2)}$, $H_{(4,0,3)}$ are satisfied on the nose
with our definitions of $H_c$ functions.
To establish the remaining recurrences, we need to show that the correponding expressions, which are linear combinations of various $S$s, are equal to $0$.
To this end, we define the known relations satisfied by the $S$s, namely the \eqref{eqn:m10r1}--\eqref{eqn:m10r4} of the mod $10$ variety:
\begin{python}
def R1(a,b,c,d):
    return(S(a,b,c,d)-S(a+1,b-1,c,d)-z*q^(a+1)*S(a+2,b,c-1,d))

def R2(a,b,c,d):
    return(S(a,b,c,d)-S(a,b,c+1,d-1)-q^(c+1)*S(a-1,b,c+2,d))

def R3(a,b,c,d):
    return(S(a,b,c,d)-S(a,b,c,d+1)-q*S(a,b+1,c,d+1)
            +q*S(a,b+1,c,d+2)-q^(c+d+2)*S(a-1,b-1,c+2,d+2))

def R4(a,b,c,d):
    return(S(a,b,c,d)-S(a,b+1,c,d)-q*S(a,b+1,c,d+1)
            +q*S(a,b+2,c,d+1)-z*q^(a+b+2)*S(a+2,b+2,c-1,d-1))
\end{python}
For compositions $(4,1,2)$, $(3,2,2)$, $(4,3,0)$ and $(3,3,1)$, we
have stored linear combinations of these relations in corresponding text files.
For instance, the file \texttt{412.txt} (see the ancillary files related to the present 
paper on \texttt{arXiv}) begins:
\begin{python}
-z*q*R1(0, 2, 0, 1)
+z*q^2*R1(0, 2, 1, 1)
+z*q^3*R1(-1, 2, 2, 2)
-z*q^2*R1(0, 2, 0, 2)
-z*q^2*R1(2, 2, 0, 0)
...
\end{python}
and has $29$ terms, see \eqref{eqn:rec412lincomb}.
The files for $(3,3,1)$ and $(4,3,0)$ each involve more than $600$ such terms, and for $(3,2,2)$
we have about $300$ terms.
The following function reads such a file, parses it into an algebraic expression understandable
by SAGE, and evaluates $R_1$--$R_4$ into linear combinations of appropriate $S$s.
\begin{python}
def read_rel(filename):
    f   = open(filename, "r")
    rel = (f.read()).replace('\n','').replace('\r','')
    return sage_eval(rel,locals=globals())
\end{python}
Now, to check the recurrence for $(4,1,2)$, we do --
\begin{python}
rel412 = read_rel("412.txt")
print(expand(rel412))
print("=======")
print(expand(checkrec((4,1,2))))
print("=======")
print(expand(rel412-checkrec((4,1,2))))
\end{python}
and the computer (happily) responds --
\begin{python}
q^4*z^2*S(3, 1, 0, 1) + q^3*z^2*S(3, 1, 0, 0) - q^3*z^2*S(3, 0, 0, 0) 
- q^3*z^2*S(2, 0, 1, 1) - q^2*z*S(2, 1, 1, 1) - q^2*z*S(1, 1, 1, 1) 
+ q*z*S(3, 0, 0, 0) - q*z*S(2, 1, 0, 0) + q*z*S(2, 0, 0, 1) + q*z*S(1, 0, 0, 1) 
- S(2, 0, 0, 1) + S(1, 1, 0, 1) + S(1, 0, 0, 0) - S(0, 1, 0, 0)
=======
q^4*z^2*S(3, 1, 0, 1) + q^3*z^2*S(3, 1, 0, 0) - q^3*z^2*S(3, 0, 0, 0) 
- q^3*z^2*S(2, 0, 1, 1) - q^2*z*S(2, 1, 1, 1) - q^2*z*S(1, 1, 1, 1) 
+ q*z*S(3, 0, 0, 0) - q*z*S(2, 1, 0, 0) + q*z*S(2, 0, 0, 1) + q*z*S(1, 0, 0, 1) 
- S(2, 0, 0, 1) + S(1, 1, 0, 1) + S(1, 0, 0, 0) - S(0, 1, 0, 0)
=======
0
\end{python}
establishing that the required recurrence is indeed a consequence of the known relations.
Same process can be repeated with remaining compositions $(3,2,2)$, $(4,3,0)$ and $(3,3,1)$.


\providecommand{\oldpreprint}[2]{\textsf{arXiv:\mbox{#2}/#1}}\providecommand{\preprint}[2]{\textsf{arXiv:#1
  [\mbox{#2}]}}

\end{document}